\documentclass{ws-jta}

\usepackage{amsmath}
\usepackage{amssymb}
\usepackage{amsbsy}
\usepackage{mathpazo}
\usepackage{textcomp}
\usepackage{color}
\usepackage{scrdate}
\usepackage{scrtime}
\usepackage{units, enumerate, array}
\usepackage[matrix,arrow,curve]{xy}

\numberwithin{equation}{section}
\renewcommand{\thetable}{\theequation}

\newtheorem{question}[equation]{Question}
\newtheorem{notations}[equation]{Notations}

\begin{document}

\markboth{U. Koschorke}{FIXED POINTS AND COINCIDENCES IN TORUS BUNDLES}

\catchline{}{}{}{}{}

\title{FIXED POINTS AND COINCIDENCES IN TORUS BUNDLES}

\author{ULRICH KOSCHORKE}

\address{
FB 6 (Mathematik), Emmy-Noether-Campus, Universit\"{a}t\\
D 57068 SIEGEN, Germany\\
koschorke@mathematik.uni-siegen.de}

\maketitle

\begin{history}
\received{(Day Month Year)}
\revised{(Day Month Year)}
\end{history}

\begin{abstract}
Minimum numbers of fixed points or of coincidence components (realized by maps in given homotopy classes) are the principal objects of study in topological fixed point and coincidence theory. In this paper we investigate fiberwise analoga and represent a general approach e.g. to the question when two maps can be deformed until they are coincidence free. Our method involves normal bordism theory, a certain pathspace \(E_B\) and a natural generalization of Nielsen numbers.

As an illustration we determine the minimum numbers for all maps between torus bundles of arbitrary (possibly different) dimensions over spheres and, in particular, over the unit circle. Our results are based on a careful analysis of the geometry of generic coincidence manifolds. They allow also a simple algebraic description in terms of the Reidemeister invariant (a certain selfmap of an abelian group) and its orbit behavior (e.g. the number of odd order orbits which capture certain nonorientability phenomena). We carry out several explicit sample computations, e.g. for fixed points in \((S^1)^2\)-bundles. In particular, we obtain existence criteria for fixed point free fiberwise maps.
\end{abstract}




\section{Introduction and discussion of results} 
\label{sec1}

The principal question of topological fixed point theory can be phrased as follows (cf. [B], p. 9).\\

\textit{Given a selfmap \(f\) of a (connected) topological space \(M\), what is the minimum number \(\mathrm{MF}(f)\) of fixed points among all the maps homotopic to \(f\)?}\\

Soon after the groundbreaking work of S. Lefschetz [L] appeared, the Danish mathematician J. Nielsen made a decisive contribution (cf. [N1], p. 256, and [N2], p. 289): he introduced a very natural equivalence relation among the fixed points of \(f\). Counting the ``essential'' equivalence classes one then obtains the Nielsen number \(\mathrm{N}(f)\) of \(f\) which is a lower bound for \(\mathrm{MF}(f)\). Actually already in the early 1940's it was known that these two numbers are equal whenever \(M\) is a compact manifold of dimension \(m\neq 2\) or a compact surface with nonnegative Euler characteristic \(\chi(M)\). So it came as a surprise when B. Jiang was able to prove in 1984/85 that \(\,\,\mathrm{MF}(f)-\mathrm{N}(f)\,\) can be strictly positive, cf. [Ji1], [Ji2] (and even arbitrarily large, cf. [Z], [Ke], [Ji3]) for suitable selfmaps of any surface with \(\chi(M)<0\).\\

It is natural to extend these studies in two directions.

First one may investigate the fixed point behaviour of fiberwise selfmaps of a fibration and, in particular, those aspects which remain preserved under fiberwise homotopies. This was done e.g. by A. Dold and resulted in the construction of his fixed point index (cf. [D]; see also [Je1], [Je2]); different approaches were used more recently e.g. in [HKW], [GNS] and [KW].

Secondly we can extend the whole discussion to pairs of maps \(f_1, f_2: M\to N\) and their coincidence sets
\begin{equation}\label{1.1}
 \mathrm{C}(f_1,f_2)=\left\{x\in M | f_1(x)=f_2(x)\right\}.
\end{equation}
This stimulated an enormous amount of research in recent years (see e.g. [BGZ], [DG], [GR1+2], [Ko1,...,5], [KR] and others) and includes the fixed point question as the special case where \(M=N\) and \(f_2\) is the identity map \(\mathrm{id}\).\\

Recently also coincidences of fiberwise maps
\begin{equation}\label{1.2}
  \xymatrix{
	f_1\; , \,f_2\; : \!\! & M \ar[rr] \ar[dr]_-{p_M} && N \ar[ld]^-{p_N}\\
	&&B&\\
  }
\end{equation}
have attracted increased attention. Here \(p_M\) and \(p_N\) are smooth fiber bundles (with typical fibers \(F_M\) and \(F_N\) of dimensions \(m\geq 0\) and \(n\geq 0\), resp.) over the closed connected \(b\)-dimensional manifold \(B\), the total space \(M\) is also closed and the maps make the diagram commute.

One of the principal problems is to determine the \textbf{\textit{minimum number}} \(\mathrm{\mathbf{MC}}_\mathbf{B}\!\mathbf{(f_1,f_2)}\)\textbf{\textit{ of coincidence points}} among all the pairs \((f_1', f_2')\) of maps such that \(f_i'\) is fiberwise homotopic \(f_i,\, i=1,2.\) Now, since we allow the dimension of \(M\) and \(N\) to differ, generic coincidence sets may well be higher dimensional submanifolds of \(M\), and \(\mathrm{MC}_B(f_1,f_2)\) need not be finite. Thus it seems more interesting to study the (finite!) \textit{\textbf{minimum number }}\(\mathrm{\mathbf{MCC}}_\mathbf{B}\!\mathbf{(f_1,f_2)}\) \textbf{\textit{ of pathcomponents of coincidence subspaces}} of \(M\) among all pairs fiberwise homotopic to \((f_1,f_2)\).

It is particulary important to know when the minimum numbers vanish, i.e. when the maps \(f_1\) and \(f_2\) can be deformed away from one another by fiberwise homotopies. In this case we say that the pair \((f_1,f_2)\) is \textit{\textbf{loose over \(\mathit{\mathbf{B}}\)}}.

A strong tool is the normal bordism class
\begin{equation}\label{1.3}
 \widetilde{\omega}_B (f_1,f_2)=\left[C,\widetilde{g},\overline{g}\right]\in  \Omega_{m+b-n}\left(E_B(f_1,f_2);\widetilde{\varphi}\right)
\end{equation}
of coincidence data introduced in [GK]. It generalizes (and often sharpens) the fixed point index of classical Nielsen theory and the strongest (``universal'') version of Dold's (fiberwise) fixed point index (compare [GK]).

Our first two coincidence data are best described by the commuting diagram
\begin{equation}\label{1.4}
  \xymatrix{
	&**[r]\ar[d]^-{\mathrm{pr}}\save{%
	 E\small{E_B(f_1,f_2):=\left\{(x,\theta)\in M\times P(N) \,\left|\,
	  \begin{matrix} p_{N}\, \circ\,\theta\,\equiv\, p_M(x);\\ \theta(0)=f_1(x),\\
	\,\theta(1)=f_2(x)\end{matrix}\right.\right\}}\restore} \\
	C \ar[ur]^-{\widetilde{g}} \ar[r]_-{g=\mathrm{incl}}&  M
  }
\end{equation}~\\
Here \(P(N)\) denotes the space of all continuous paths \(\,\theta: [0,1] \to N\,\); moreover \(C\) is the (generic) coincidence manifold (of a smooth transverse approximation of \((f_1,f_2)\) if necessary, cf. [GK], 1.4), and \(\widetilde{g}\) lifts the inclusion map \(g\) in the Hurewicz fibration \(\mathrm{pr}\) by putting 
\[ \widetilde{g}(x):=\left(x, \text{constant path at }f_1(x)=f_2(x)\right)\] for \(x\in C\). We obtain the third coincidence datum \(\overline{g}\) by extracting the (stabilized) vector bundle isomorphism
\[TC\oplus f_1^\ast(TN)|C \cong \left.\left(TM\oplus p_M^\ast(TB)\right)\right|C\] from the geometry of \(C\subset M\) and using it to express the stable normal bundle of the manifold \(C\) as a pullback, via \(\widetilde{g}\), of the virtual vector bundle
\begin{equation}\label{1.5}
 \widetilde{\varphi}:=pr^\ast(\varphi)
\end{equation}
over \(E_B(f_1,f_2)\) (where
\begin{equation}\label{1.6}
 \varphi:= f_1^\ast(TN) -TM-p_M^\ast(TB)
\end{equation}
over \(M\); for details see [GK], 1.4-1.11 or compare [Ko2]).\\

In general the pathspace \(E_B(f_1,f_2)\) has a very rich topology, and the lifting \(\widetilde{g}\) captures much more information than the inclusion map \(g\) (which reflects only ``how the coincidence manifold \(C\) lies in \(M\)``) does. Already the decomposition of \(E_B(f_1,f_2)\) into its pathcomponents induces the decomposition
\begin{equation}\label{1.7}
 C=\coprod_{Q\in\pi_0\left(E_B(f_1,f_2)\right)} C_Q
\end{equation}
into the closed manifolds \(C_Q:=\widetilde{g}^{-1}(Q)\) (which are almost all empty).
\begin{definition} \label{1.8}
 We call a pathcomponent \(Q\) of \(E_B(f_1,f_2)\) \textbf{\textit{essential}} if the triple \(\left(C_Q,\, \widetilde{g}|C_Q,\, \overline{g}|\right)\) of restricted coincidence data represents a nontrivial normal bordism class in \(\,\,\Omega_\ast\!\left(Q;\,\widetilde{\varphi}|Q\right).\)

The \textbf{\textit{Nielsen number }}\(\mathbf{N}_\mathbf{B}\!\mathbf{(f_1,f_2)}\) is the number of essential pathcomponents of \(E_B(f_1,f_2)\).

The \textbf{\textit{geometric Reidemeister set \(\mathrm{\mathbf{R}}_\mathbf{B}\!\mathbf{(f_1,f_2)}\)}} (and \textbf{\textit{the Reidemeister number}}, resp.) of the pair \((f_1,f_2)\) is the set \(\pi_0\left(E_B(f_1,f_2)\right)\) of all pathcomponents of \(E_B(f_1,f_2)\) (and its cardinality, resp.).
\end{definition}
Clearly \(\mathrm{N_B}(f_1,f_2)\leq \# \mathrm{R_B}(f_1,f_2)\) and the Nielsen number (but \textit{not} necessarily the Reidemeister number) is always finite. Moreover it can be shown that both numbers depend only on the fiberwise homotopy classes of \(f_1\) and \(f_2\) and that
\begin{equation}\label{1.9}
 \mathrm{N_B}(f_1,f_2)\leq \mathrm{MCC}_B(f_1,f_2)\leq\mathrm{MC}_B(f_1,f_2)
\end{equation}
(compare [Ko2]).
\begin{example}[\textit{classical Nielsen fixed point theory in manifolds}]\label{1.10}\\
 Here \(B\) consists only of one point and \((f_1,f_2)=(f,\mathrm{id})\) where \(f\) is a selfmap of a connected manifold \(M=N\). There is a bijection from the classical (''algebraic``) Reidemeister set \(\pi_1(M)\diagup\!\!\sim\,\,\ \)  onto \(R_B(f,\mathrm{id})=\pi_0\left(E_B(f,\mathrm{id})\right)\) and
\[\widetilde{\omega}_B(f,\mathrm{id})\, \in\, \, \Omega_0\left(E_B(f,\mathrm{id});\,\widetilde{\varphi}\right)=\bigoplus_{Q\in \mathrm{R_B}(f, \mathrm{id})}\mathbb{Z}\]
records the indices of all the Nielsen fixed point classes of \(f\) (cf. [Ko2]). In particular, our definition (\ref{1.1}) agrees with the classical definition of Nielsen and Reidemeister numbers (cf. [B]). Similarly, \(\mathrm{MC}_B(f,\mathrm{id})\) is just the classical minimum number \(\mathrm{MF}(f)\) of fixed points (cf. [Br]).\(\hfill\Box\)
\end{example}
In general the lower bound \(\mathrm{N_B}(f_1,f_2)\) fails often to be also equal to \(\mathrm{MCC}_B(f_1,f_2)\) (for infinitely many counterexamples involving e.g. maps of the form \(f_1,f_2:S^{2n-1}\to S^n\) see [Ko2], 1.17). This lead to the construction of a ''nonstabilized`` version \(\omega_B^{\#}(f_1,f_2)\) of our \(\omega\)-invariant and to a sharper Nielsen number \(\mathrm{N_B}^{\#}(f_1,f_2)\) which agrees with the minimum number \(\mathrm{MCC}_B(f_1,f_2)\) at least for all maps between spheres (cf. [Ko4]). However, already for maps into real projective spaces new discrepancies appear which are related to subtle problems concerning nonvanishing Kervaire invariants or divisibility questions for Whitehead products or Hopf invariants (cf. [KR]).

In contrast, in this paper we will study a setting where our original Nielsen number \(\mathrm{N_B}(f_1,f_2)\) fully determines the minimum numbers of coincidence components or points.
\begin{definition}\label{1.11}
 A \textbf{\textit{linear torus bundle}} is a smooth fiber bundle \(p: M\to B\) with typical fiber a torus \(T^m=\left(S^1\right)^m\) and stucture group \(\mathrm{GL}\left(m,\mathbb{Z}\right)\) (which acts in the standard fashion on \(\left(\mathbb{R}^m,\mathbb{Z}^m\right)\) and hence on \(T^m\)).
\end{definition}
In the remainder of this paper \(p_M: M\to B\) and \(p_N:N\to B\) will always denote linear torus bundles with typical fibers \(T^m\) and \(T^n\), resp., of (possibly different) dimensions \(m,n\geq 0\), and \(f_1,f_2:M\to N\) will be fiberwise maps. Every fiber of \(p_N\) has a natural abelian Lie group structure which makes it isomorphic to \(T^n\) and which we write additively. Similary we can add and subtract fiberwise maps into \(N\).

Choose isomorphisms \(T^m\cong F_M:=p_M^{-1}\{\ast\}\) and \(F_N:=p_N^{-1}\{\ast\}\cong T^n\) for the fibers over some base point \(\ast\) of \(B\). When restricted to these fibers, \(f_1-f_2\) then induces the composite homomorphism
\begin{equation}\label{1.12}
\xymatrix@=2cm{
 \overline{L}\,\;: \,\;\mathbb{Z}^m=\pi_1\left(T^m\right)\,\,\ar[r]^-{\left(f_1|-f_2|\right)_\ast} &\,\,\pi_1\left(T^n\right)=\mathbb{Z}^n
}
\end{equation}
between fundamental groups which extends to yield a linear map from \(\mathbb{R}^m\) to \(\mathbb{R}^n\) (also denoted by \(\overline{L}\)).
\begin{theorem}\label{1.13}
 Assume that the base space \(B\) consists of a single point (of dimension \(b=0\)) or is the sphere \(S^b\) of dimension \(b\geq 1\). Then:
\begin{enumerate}[(i)]
 \item \begin{alignat*}{2}
&\mathrm{MCC}_B(f_1,f_2)=\mathrm{N}_B(f_1,f_2)&\qquad&\text{and}\\
&\mathrm{MC}_B(f_1,f_2)=
\begin{cases}
 \mathrm{N}_B(f_1,f_2) &\text{if }\, \mathrm{N}_B(f_1,f_2)=0\, \text{ or } \,m+b=n; \\
\infty&\text{else.}
\end{cases}
\end{alignat*}
\item Assume in addition that \(b\neq1\). Then\[
\mathrm{N}_B(f_1,f_2)=\left|\det\left(\overline{u}_1,\ldots,\overline{u}_n\right)\right|\]
where the column vectors \(\overline{u}_1,\ldots,\overline{u}_n\) of the indicated \(n\times n\)-matrix generate \(\overline{L}\left(\mathbb{Z}^m\right)\). In particular, \(\mathrm{MCC}_B(f_1,f_2)\) (or, equivalently, \(\mathrm{MC}_B(f_1,f_2)\))  vanishes if and only if the linear map \(\overline{L}:\mathbb{R}^m\to\mathbb{R}^n\) is not onto. Moreover the Reidemeister number is given by\[
\#\mathrm{R}_B(f_1,f_2)=\#\left(\mathbb{Z}^n\diagup\overline{L}\left(\mathbb{Z}^m\right)\right)=
\begin{cases}
  \mathrm{N}_B(f_1,f_2)&\text{if }\, \mathrm{N}_B(f_1,f_2)\neq 0;\\
\infty& \text{if }\, \mathrm{N}_B(f_1,f_2)= 0.
\end{cases}
\]
It follows that the invariants \(\mathrm{MC}_B(f_1,f_2)\) and \(\#\mathrm{R}_B(f_1,f_2)\) can only take the values \(\mathrm{N}_B(f_1,f_2)\) and \(\infty\); they agree precisely when\[
\mathrm{N}_B(f_1,f_2)\neq 0=m+b-n\]
(i.e. in the only case when they are both finite).
\end{enumerate}
\end{theorem}
For \(b=0\) (and \(b\geq 2\),resp.) all this will be proved in sections \ref{sec2} (and \ref{sec3}, resp.) below.\\

Next let us focus on the remaining case \(\,\,B=S^1=I\diagup\! 0\sim 1\,\,\) which turns out to be particularly interesting. There exist invertible matrices \(\overline{A}_M\in \mathrm{GL}\left(m,\mathbb{Z}\right)\) and \(\overline{A}_N\in \mathrm{GL}\left(n,\mathbb{Z}\right)\) (which induce automorphisms \(A_M\) and \(A_N\), resp., of tori) such that up to fiberwise isomorphisms
\begin{equation}\label{1.14}
 \begin{split}
  M &= \left(T^m\times I\right)\diagup(x,1)\sim\left(A_M(x),0\right)\quad\text{and}\\
N &= \left(T^n\times I\right)\diagup(u,1)\sim\left(A_N(u),0\right).
 \end{split}
\end{equation}
Let \(s_{0 M}\) (and \(s_{0 N}\), resp.) denote the zero section of \(p_M\) (and \(p_N\), resp.) defined by \(\left[t\right]\to\left[(0,t)\right]\), and put \(f_0:=s_{0 N}\circ p_M\).

Given a pair \(f_1,f_2:M\to N\) of fiberwise maps, it has the same minimum and Nielsen numbers as the pair formed by \(f:=f_1-f_2\) and \(f_0\). The fiberwise homotopy class of \(f\) is fully determined by the linear map \(\overline{L}:\left(\mathbb{R}^m,\mathbb{Z}^m\right)\to\left(\mathbb{R}^n,\mathbb{Z}^n\right)\) (which reflects the behaviour of \(f\) along a single ''vertical`` fiber, cf. \ref{1.12}) and the ''horizontal`` datum
\begin{equation}\label{1.15}
 \left[\overline{v}\right]\in\mathbb{Z}^n\diagup\left(\overline{A}_N-\mathrm{id}\right)\left(\mathbb{Z}^n\right)
\end{equation}
(which classifies the section \(f\circ s_{0 M}\) of \(p_N\), see section \ref{sec4} below).

It is easy to see that
\begin{equation}\label{1.16}
 \overline{A}_N\circ\overline{L}=\overline{L}\circ\overline{A}_M.
\end{equation}
Thus \(\overline{A}_N\) preserves \(\overline{L}\left(\mathbb{Z}^m\right)\) and every representative \(\overline{v}\) of the class \(\left[\overline{v}\right]\) defines a selfmap \(\beta_{\overline{v}}\) of \(\mathbb{Z}^n\diagup\overline{L}\left(\mathbb{Z}^m\right)\) by
\begin{equation}\label{1.17}
 \beta_{\overline{v}}\left[\overline{u}\right] := \left[\overline{A}_N\left(\overline{u}-\overline{v}\right)\right], \quad \overline{u}\in\mathbb{Z}^n.
\end{equation}
The iterates of \(\beta_{\overline{v}}\) determine an action of the group \(\mathbb{Z}\) whose orbits correspond bijectively to the elements of the Reidemeister set \(\mathrm{R}_B(f_1,f_2)=\pi_0\left(E_B(f_1,f_2)\right)\) (cf. \ref{1.8}). Indeed restrict \(E_B(f_1,f_2)\) to a single fiber \(F_M\) of \(p_M\). Then \(\;\beta_{\overline{v}}\;\) describes the operation of \(\;\pi_1\left(S^1\right)=\mathbb{Z}\;\) on \[
\pi_0\left(E_B(f_1,f_2)|F_M\right)=\pi_0\left(E(f_1|F_M,f_2|F_M)\right)\,\approx\, \mathbb{Z}^n\diagup \overline{L}\left(\mathbb{Z}^m\right)\]
(compare [Ko2] 2.1, [Ko6] and section \ref{sec4} below). Each pathcomponent \(Q\) of \(E_B(f_1,f_2)\) corresponds to the orbit which consists of the pathcomponents of the intersection \(\,Q\,\cap\,\left( E_B(f_1,f_2)|F_M\right)\).
The number of orbits of a given cardinality depends only on the class of \(\,\overline{v}\,\) in  \(\,\,\mathbb{Z}^n\diagup \left(\overline{A}_N - \mathrm{id}\right)\left(\mathbb{Z}^n\right)\,\,\) (see also lemma \ref{4.8} below).
\begin{theorem}\label{1.18}
 For every pair of fiberwise maps \(f_1,f_2:M\to N\) over \(S^1\) the minimum, Nielsen and Reidemeister numbers are given as follows.\\\\
\textbf{Case 0: }\(\ \mathrm{\mathbf{dim}}\,\mathbf{\overline{L}\left(\mathbb{R}^m\right)=n.}\) Here \(\,\mathrm{MC}_B(f_1,f_2)=\infty\,\) and \[\mathrm{MCC}_B(f_1,f_2)=\mathrm{N}_B(f_1,f_2)=\#\mathrm{R}_B(f_1,f_2)\]
equals the number of orbits of (the \(\mathbb{Z}\)-action defined by) the selfmap \(\beta_{\overline{v}}\) on the finite set \(\mathbb{Z}^n\diagup\overline{L}\left(\mathbb{Z}^m\right)\).\\\\
\textbf{Case 1: }\(\ \mathrm{\mathbf{dim}}\,\mathbf{\overline{L}\left(\mathbb{R}^m\right)=n-1.}\) Here \[
\mathrm{MC}_B(f_1,f_2)=
\begin{cases}
\mathrm{N}_B(f_1,f_2)&\text{if }\, \mathrm{N}_B(f_1,f_2)=0\,\text{ or }\,m<n;\\
\infty&\text{if }\, \mathrm{N}_B(f_1,f_2)\neq 0\,\text{ and }\,m\geq n.
\end{cases}
\]
Moreover \(\overline{A}_N\) induces an automorphism \(\overline{A}_{N\ast}\) of the quotient group
\begin{equation}\label{1.19}
\mathbb{Z}^n\diagup\left(\mathbb{Z}^n\cap\overline{L}\left(\mathbb{R}^m\right)\right)\,\,\cong\,\,\mathbb{Z}
\end{equation}
i.e. \(\overline{A}_{N\ast}=\mathrm{a}\cdot\mathrm{identity map}\) where \(a=\pm 1\).\\
\textbf{Subcase 1+ : }\(\ \mathbf{a=+1}\). Here\[
\mathrm{MCC}_B(f_1,f_2)=\mathrm{N}_B(f_1,f_2)= \left|\det\left(\overline{v},\overline{w}_1,\ldots,\overline{w}_{n-1}\right)\right|\]
where the vectors \(\overline{w}_1,\ldots,\overline{w}_{n-1}\in\mathbb{Z}^n\) generate \(\overline{L}\left(\mathbb{Z}^m\right)\). Furthermore\[
\#\mathrm{R}_B(f_1,f_2)=
\begin{cases}
 \mathrm{N}_B(f_1,f_2)&\text{if }\,\overline{v}\notin\overline{L}\left(\mathbb{R}^m\right);\\
\infty&\text{else}. 
\end{cases}
\]
\textbf{Subcase \(\mathbf{1-}\): }\(\ \mathbf{a=-1}\). Here the Reidemeister set \(\mathrm{R}_B(f_1,f_2)\) is infinite. Furthermore the quotient map \(\mathbb{Z}^n\to\mathbb{Z}\) defined by \ref{1.19} induces the homomorphism\[
\xymatrix@=1.2cm{
\mathrm{quot}\,\,:\,\, \mathbb{Z}^n\diagup\left(\overline{A}_N -\mathrm{id}\right)\left(\mathbb{Z}^n\right) \,\ar[r] & \, \mathbb{Z}_2.}\]
If \(\mathrm{quot}\left(\left[\overline{v}\right]\right)\neq 0\), then \((f_1,f_2)\) is loose over \(S^1\), i.e. \[
\mathrm{MCC}_B(f_1,f_2)= \mathrm{N}_B(f_1,f_2)=0. \]
If \(\mathrm{quot}\left(\left[\overline{v}\right]\right)= 0\), we may choose a representive \(\overline{v}\) of \(\left[\overline{v}\right]\) which lies in \(\overline{L}\left(\mathbb{R}^m\right)\). Then \(\beta_{\overline{v}}\) (cf. \ref{1.17}) restricts to yield the selfmap \(\beta_{\overline{v}}|\) of the finite set  \(\,\left(\mathbb{Z}^n\cap\overline{L}\left(\mathbb{R}^m\right)\right)\diagup \overline{L}\left(\mathbb{Z}^m\right)\); moreover \[
\mathrm{MCC}_B(f_1,f_2)=\mathrm{N}_B(f_1,f_2)\]
equals the number of \textnormal{\textbf{odd}} order orbits of (the \(\mathbb{Z}\)-action determined by) \(\ \beta_{\overline{v}}|\).\\\\
\textbf{Case 2: }\(\ \mathrm{\mathbf{dim}}\,\mathbf{\overline{L}\left(\mathbb{R}^m\right)\leq n-2.}\) Here \((f_1,f_2)\) is loose and \[
\mathrm{MC}_B(f_1,f_2)=\mathrm{MCC}_B(f_1,f_2)=\mathrm{N}_B(f_1,f_2)=0.\]
\end{theorem}
\begin{example}[\(m=n=2,\, \overline{A}_M=\overline{A}_N= \begin{pmatrix}
0 & -1 \\ 
1 & 0
\end{pmatrix}\)]\label{1.20neu2}\\
 We identify \(\,\overline{z}=(\overline{z}_1,\,\overline{z}_2)\in\mathbb{Z}^2\,\) with the Gau\ss{} integer \(\,\overline{z}_1+i\,\overline{z}_2\in\mathbb{Z}\oplus i\,\mathbb{Z}\subset\mathbb{C}\,\) so that the \(90^\circ\) degree rotation \(\,\overline{A}_M=\overline{A}_N\,\) amounts to multiplication with the complex number \(i\). Then fiberwise maps \(\,f\,:\,M\,\to\, N\,\) are classified (up to fiberwise homotopy) by pairs \(\,(\overline{L}, [\overline{v}])\,\) where \(\,\overline{L}\,:\, \mathbb{Z}^2\,\to\,\mathbb{Z}^2\,\) is complex linear and the class \([\overline{v}]\) of \(\,\overline{v}=(\overline{v}_1,\overline{v}_2)\in\mathbb{Z}^2\,\) corresponds to \(\,[\overline{v}_1+\overline{v}_2]\in\mathbb{Z}_2\,\) under the isomorphism \(\,\,\mathbb{Z}^2\diagup(i-1)\mathbb{Z}^2\,\,\cong\,\,\mathbb{Z}_2\,\,\) (cf. \ref{1.15}, \ref{1.16} and proposition \ref{4.6}). At the end of section \ref{sec4} below we will prove\vspace{-4mm}
\begin{proposition}\label{1.21neu2}
Given fiberwise maps \(f_1,\,f_2\;:\,M\,\to\,N\), consider the classifying pair \(\,\left(\overline{L},[\overline{v}]\right)\,\) for \(f=f_1-f_2\). Write \(\,\overline{L}(1)=k+i\,l\,\) and \(\,k^2+l^2=4\,q+r\,\) (with integers \(k,l,q,r\) satisfying \(q\geq0\,\) and \(\,r=0,1\) or \(2\)).\\
Then\[
\mathrm{MCC}_B(f_1,f_2)\,=\,
\begin{cases}
0&\text{if }\, k^2+l^2=0;\\
q >0& \text{if }\, k^2+l^2=4\,q>0\,\text{ and }\,\overline{v}_1\not\equiv\overline{v}_2(2);\\
q+2 &\text{if }\, k^2+l^2>0 \,\text{ is even and }\,\overline{v}_1\equiv\overline{v}_2(2);\\
q+1&\text{else}.
\end{cases}
\]

More precisely, if \(\,\overline{L}\neq0\,\), then the selfmap \(\,\beta_{\overline{v}}\,\) on \(\,\,G:=\mathbb{Z}^2\diagup\overline{L}\left(\mathbb{Z}^2\right)\,\,\) has only orbits of order 1, 2 and 4, resp. Their numbers \(\nu_1,\nu_2\) and \(\nu_4\), resp., as well as \(\,\nu:=\nu_1+\nu_2+\nu_4\,\) depend only on the cardinality \(\,k^2+l^2\,\) of \(G\) and on the parity of \(\,\overline{v}_1+\overline{v}_2\,\), and are listed in table \ref{1.22neu2}.

In particular, all odd order orbits consist of a single fixed point. Their number \(\nu_1\) has the same parity as \(\,\#G=k^2+l^2\) and can take only the values \(0,\ 1\) and \(2\).\vspace{-4mm}
\refstepcounter{equation}
\begin{table}[h]\label{1.22neu2}
\begin{tabular}{l|l|l|l}
 & \(k^2+l^2=4\,q>0\) & \(k^2+l^2=4\,q+1\) & \(k^2+l^2=4\,q+2\) \\ \hline
\(\overline{v}_1\equiv\overline{v}_2(2)\) & \(2,\ 1,\ q-1;\ q+2\) & \(1,\ 0,\ q;\ q+1\) & \(2,\ 0,\ q;\ q+2\) \\ \hline
\(\overline{v}_1\not\equiv\overline{v}_2(2)\) & \(0,\ 0,\ q;\quad\ \ \ \ q\) & \(1,\ 0,\ q;\ q+1\) & \(0,\ 1,\ q;\ q+1\)\\\hline
\end{tabular}
\caption{The numbers of orbits of \(\beta_{\overline{v}}\) with a given cardinality (1, 2, 4 or arbitrary, resp.)}
\end{table}
\end{proposition}
Now let us replace \(\overline{L},\,\overline{v}\) and the bundles \(M,\,N\) by \(\,\overline{L}\oplus0\,:\, \mathbb{Z}^3\,\to\,\mathbb{Z}^3,\, (\overline{v},0)\in\mathbb{Z}^3\,\) and the fiberwise product of \(M=N\) with an \(S^1\)-bundle \(R\) over \(S^1\), resp. If \(R\) is the torus \(S^1\times S^1\) (with a standard projection) we are in subcase 1+ and all Nielsen classes become inessential. However, if \(R\) is the Klein bottle \(K\) (subcase 1-) precisely those (at most two) essential Nielsen classes survive which correspond to odd order orbits - in spite of the extra space for possible deformations which the transition from \(N\) to \(\, N\,\times_B\,K\,\) provides. In particular, \(\, f_1\oplus 0\,\) can be deformed away from \(\, f_2\oplus 0\,\) if and only if \(\,k+l\,\equiv\,0\,\not\equiv\,\overline{v}_1+\overline{v}_2\ (2)\,\) or \(\, k=l=0\).\(\hfill\Box\)
\end{example}
We can avoid the many case distinctions and express the description of Nielsen numbers in theorems \ref{1.13}(ii) and \ref{1.18} in a way which may look more coherent and elegant (but which is possibly not so useful for direct concrete calculations and also obscures a little the special role which e.g. odd order orbits play).
\begin{definition}\label{1.20neu}
 \begin{enumerate}[(i)]
  \item We call a map \(\,k: G\, \rightarrow \,G'\,\) between abelian groups \textbf{\textit{affine}} (or \textbf{affine isomorphism}, resp.) if it is the sum of a group homomorphism (or isomorphism, resp.) and a constant map.
\item Given a base point \(\ast\) of \(B\), consider pairs \((G,\beta)\) where \(G\) is an abelian group and \(\beta\) is an action of the group \(\pi_1(B,\ast)\) on \(G\) by affine automorphisms. We call two such pairs \((G,\beta)\) and \((G',\beta')\) \textbf{equivalent} if there exists an affine isomorphism \(\,k: G\,\rightarrow \,G'\,\) such that \(\,k\left(\beta(y,g)\right)=\beta'\left(y,k(g)\right)\) for all \(y\in \pi_1(B,\ast)\) and \(g\in G\).

The resulting set of equivalence classes is denoted by \(\mathcal{R}_{B}\).
 \end{enumerate}
\end{definition}

Clearly the number of orbits of \(\beta\) of a given cardinality depends only on the equivalence class of \((G,\beta)\), and so does the rank of \(G\). In particular, there are welldefined functions
\begin{equation}\label{1.21neu}
\xymatrix@=1.2cm{
 \nu_{\mathrm{odd}}\,,\,\nu_{\mathrm{even}}\,,\,\nu_{\infty}\,:\,\, \mathcal{R}_{B}\,\ar[r] &\,\,\mathbb{N}\cup\{\infty\}}
\end{equation}
which count all orbits of odd, even and infinite order, resp. When we ''reduce mod \(\infty\)``, i.e. when we replace the value \(\infty\) by \(0\) while leaving all finite values unchanged, we obtain the functions
\begin{equation}\label{1.21neu'}\tag{\ref{1.21neu}'}
\xymatrix@=1.2cm{
 \nu'_{\mathrm{odd}}\,,\,\nu'_{\mathrm{even}}\,,\,\nu'_{\infty}\,:\,\, \mathcal{R}_{B}\,\,\ar[r] &\,\,\mathbb{N}}.
\end{equation}

Observe that direct sums make \(\,\mathcal{R}_{B,\ast}\,\) into a monoid. Note also that changes of basepoints in \(B\) lead to isomorphism (of fundamental groups and hence of these monoids) which preserve orbit numbers etc. Thus we will often drop the basepoints from our notation.

Now, given fiberwise maps \(\,f_1,f_2\,\) (cf. \ref{1.2}) between torus bundles, the composite of projections
\begin{equation}\label{1.25neu3}
\xymatrix@=1.2cm{
 p_M\circ \mathrm{pr}\;:\;E_B(f_1,f_2)\,\ar[r] &\, B
}
\end{equation}
(cf. \ref{1.2} and \ref{1.4}) is a Serre fiber map with fiber \(\;E(f_1|F_M\,,\;f_2|F_M)\;\). The end of the resulting fiber homotopy sequence\[
\xymatrix{
\ldots \ar[r]\, &\, \pi_1(B)\, \ar@{-->}[r]^-{\beta} &\,\pi_0\left(E(f_1|F_M,\,f_2|F_M)\right)\,\ar[r]& \,\pi_0\left(E_B(f_1,f_2)\right)\, \ar[r]&\,0
}\]
leads to a description of the elements in the geometric Reidemeister set \(\;\mathrm{R}_B(f_1,f_2)=\pi_0\left(E_B(f_1,f_2)\right)\;\) (cf. definition \ref{1.8}) as orbits of a group action \(\beta\) on the geometric Reidemeister set of the restricted pair \(\;(f_1|F_M,\,f_2|F_M)\). Since the fiber \(F_N\) has an abelian fundamental group the algebraic interpretation of this Reidemeister set yields a group structure (see also [Ko6]). We obtain an isomorphism\[
\pi_0\left(E(f_1|F_M\,,\;f_2|F_M)\right)\;\cong\;\mathbb{Z}^n\diagup\overline{L}\left(\mathbb{Z}^m\right)\]
(cf. \ref{1.12}) and a resulting group action (also denoted by \(\beta\)) of \(\,\pi_1(B)\,\) on \(\;\mathbb{Z}^n\diagup\overline{L}\left(\mathbb{Z}^m\right)\;\) by affine isomorphisms.

\begin{definition}\label{1.26neu3}
Given a pair \((f_1,f_2)\) of fiberwise maps between linear torus bundles over a closed connected manifold \(B\), we define its \textbf{\textit{Reidemeister  invariant}} \(\,\varrho_B(f_1,f_2)\in\mathcal{R}_B\,\) to be the equivalence class of the pair \(\,\left(G:=\mathbb{Z}^n\diagup\! \overline{L}\left(\mathbb{Z}^m\right)\,, \;\beta\right)\,\).
\end{definition}

The Reidemeister invariant is independent of all choices made in its construction and depends only on the fiberwise homotopy classes of \(f_1\) and \(f_2\). Furthermore it is compatible with the natural products (fiberwise products on the one hand and direct products on the other hand).

We obtain \(\;\varrho_B(f_1,f_2)\;\) by carrying out the Reidemeister operation of \(\,\pi_1(M)\,\) on \(\,\pi_1(F_N)\,\) (cf. [GK], definition 3.1) in two steps:
\begin{enumerate}[(i)]
 \item First let the ''vertical`` part of \(\,\pi_1(M)\,\) (i.e. the image of \(\,\pi_1(F_M)\,\)) act on \(\,\pi_1(F_N)\); the resulting ''partial`` orbit set is \(G\).
\item Then \(\,\beta\,\) captures the remaining ''horizontal`` action; the elements in a given orbit of \(\,\beta\,\) correspond to the pathcomponents of \(\,E_B(f_1,f_2)|F_M\,\) which lie in the same given pathcomponent of \(\,E_B(f_1,f_2)\).
\end{enumerate}

Thus the orbit numbers \(\;\nu'_{\mathrm{odd}},\,\nu'_{\mathrm{even}}\;\) and \(\;\nu'_{\infty}\;\) of \(\;\varrho_B(f_1,f_2)\;\) do not seem to involve \(\,\widetilde{\omega}\,\)-invariants and essentiality questions (cf. definition \ref{1.8}) at all but just reflect the crudest aspects of how the total space of \(\,p_M\circ \mathrm{pr}\,\) (cf. \ref{1.25neu3}) intersects one of the fibers. So the following result may come as a surprise.

\begin{theorem}\label{1.22neu}
 If \(\,B\,\) is a sphere (of any positive dimension) or a point we have for every pair \(\,f_1,\,f_2\,:\;M\,\rightarrow\,N\,\) of fiberwise maps between linear torus bundles over \(B\) \[
\mathrm{MCC}_B (f_1,f_2)=\mathrm{N}_B (f_1,f_2)=\nu_B\left(\varrho_B(f_1,f_2)\right)\]
where\[
\nu_B(G,\beta):=
\begin{cases}
 \left(\nu'_{\mathrm{odd}}+\nu'_{\mathrm{even}}+\nu_{\infty}\right)(G,\beta) &\text{if }\, \mathrm{rank}(G)\leq \mathrm{rank}\left(\pi_1(B)\right);\\
0&\text{else}.
\end{cases}
\]

The minimum number \(\mathrm{MC}_B(f_1,f_2)\) depends only on the Reidemeister invariant \(\varrho_B(f_1,f_2)\) and on the difference of the dimensions of \(M\) and \(N\). Moreover\[
\#\mathrm{R}_B (f_1,f_2) = \left(\nu_{\mathrm{odd}}+ \nu_{\mathrm{even}}+\nu_{\infty}\right)\left(\varrho_B(G,\beta)\right).\]
\end{theorem}
\begin{question}\label{1.28neu3}
 Is there a comparable result for other base manifolds \(\,B\,\) (where \(\,\nu_B\,\) may also involve the dimension and further aspects of \(B\))?

When (and how) can we conclude just from the topological properties of \(\,E_B(f_1,f_2)\,\) and of \(\,p_M\circ \mathrm{pr}\,\) (cf. \ref{1.25neu3}) whether a given pathcomponent \(\,Q\,\) of \(\,E_B(f_1,f_2)\,\) is essential? \(\hfill\Box\)
\end{question}

\begin{example}\label{1.23neu}
 If \(M=N\) is the fiberwise product of at least two Klein bottles over \(S^1\) (i.e. \(\,\overline{A}_M= \overline{A}_N =-\mathrm{id}\,\) on \(\mathbb{R}^n,\,n\geq 2\)) and if \(f_1\) and \(f_2\) are both fiberwise zero (i.e. \(f_1=f_2=f_0\)), then \(\nu_{\mathrm{odd}},\,\nu_{\mathrm{even}}\) and \(\nu_{\infty}\) , resp., map the Reidemeister invariant \(\,\varrho_B(f_1,f_2)=\left[\left(\mathbb{Z}^n\,,\,-\mathrm{id}\right)\right]\,\) to \(1,\,\infty\) and \(0\), resp.; therefore\[
\left(\nu'_{\mathrm{odd}}+\nu'_{\mathrm{even}}+\nu_{\infty}\right)(G,\beta)\,=\,1\,\neq\,\mathrm{N}_B(f_1,f_2)\,= \,\mathrm{MCC}_B(f_1,f_2)\,=\,0.\]
Thus the case distinction in the definition of \(\nu_B\) in theorem \ref{1.22neu} seems to be unavoidable.\(\hfill\Box\)\\

Theorems \ref{1.13}, \ref{1.18} and \ref{1.22neu} follow from the discussions in sections \ref{sec3}, \ref{sec4} and \ref{sec6} below. These involve very heavily the invariant \(\widetilde{\omega}_B(f_1,f_2)\) (i.e. a detailed analysis of the coincidence submanifold \(C\) in the domain \(M\)) and, in the process, the kernel of \(\overline{L}\). So it is all the more striking that the final results are expressed mainly in terms of data related to the target bundle \(N\) such als the cokernel of \(\overline{L}\), whereas \(M\) gets litte visibility and \(\mathrm{ker}\,\overline{L}\) plays no role at all.
\end{example}
Next consider the \(\mathrm{mod}\,2\) Hurewicz homomorphism
\begin{equation}\label{1.20}
 \mu_2: \Omega_\ast\left(E_B(f_1,f_2);\widetilde{\varphi}\right)\to H_\ast\left(E_B(f_1,f_2);\mathbb{Z}_2\right)
\end{equation}
which takes the normal bordism class \(\,\widetilde{\omega}_B(f_1,f_2)= \left[C,\widetilde{g},\overline{g}\right]\) to the image of the fundament class of \(C\) under the induced homomorphism \(\widetilde{g}_\ast\) and forgets the ''twisted framing`` \(\overline{g}\) entirely. In general coincidence theory often this means a big loss of information: singular homology (even with twisted integer coefficients) is usually far too weak to capture all essential geometric aspects (see e.g. the discussion following of (3.4) in [Ko1] for examples where \(\mu\left(\widetilde{\omega}_B(f_1,f_2)\right)\) is trivial but \(\overline{g}\) allows decisive and very subtle distinctions).

However in the setting of torus bundles we encounter a very unusual phenomenon: in sufficiently many cases \(\widetilde{g}\) turns out to be a homotopy equivalence. This is crucial in the proofs of theorems \ref{1.13} and \ref{1.18} and implies the following side result (cf. proposition \ref{3.2}, remark \ref{5.9} and theorem \ref{6.14} below).
\begin{theorem}\label{1.21}
 For all fiberwise maps \(f_1,f_2\) between linear torus bundles over \(B=\{\text{point}\}\) or \(B=S^b,\,b\geq 1\), we have: a pathcomponent \(Q\) of \(E_B(f_1,f_2)\) is essential if and only if\[
\mu_2\left(\left[C_Q,\,\widetilde{g}|C_Q,\,\overline{g}|\right]\right) = \widetilde{g}_\ast\left(\left[C_Q\right]\right) \in H_{m+b-n}\left(Q;\mathbb{Z}_2\right)\]
does not vanish (compare definition \ref{1.8}).
\end{theorem}
Thus we do not need the full power of normal bordism techniques in our special setting but homological methods suffice here to determine Nielsen (and hence minimum) numbers. As another illustration of this point let us mention the following criterion (cf. corollary \ref{6.13} below):\\
a pathcomponent \(Q\) of \(\,E_B(f_1,f_2)\,\) is essential in subcase 1\(-\ \) of theorem \ref{1.18} if and only if the projection to \(B=S^1\) induces an epimorphism from \(H_1\left(Q;\mathbb{Z}_2\right)\) to \(H_1\left(S^1;\mathbb{Z}_2\right)\).\\

In section \ref{sec5} below we will present a technique to compute the relevant normal bordism and \(\mathrm{mod}\,2\) homology groups when \(B=S^1\). This allows us also to compare \(\widetilde{\omega}_B(f_1,f_2)\) to the seemingly less complicated invariant
\begin{equation}\label{1.22}
 \omega_B(f_1,f_2)=\mathrm{pr}_\ast \left(\widetilde{\omega}_B(f_1,f_2)\right)\in \Omega_{m+1-n}\left(M;\varphi\right)
\end{equation}
which does not involve the pathspace \(E_B(f_1,f_2)\) (cf. \ref{1.4} and \ref{1.6}). It turns out that \(\mu_2\left(\widetilde{\omega}_B(f_1,f_2)\right)\in H_{m+1-n} \left(E_B(f_1,f_2)\right)\) is not only much stronger than \(\,\mu_2\left(\omega_B(f_1,f_2)\right)\) but often also easier to work with since it contains all the relevant information but no unnecessary redundancies.
\begin{example}[Fixed point theory over \(S^1\)]\label{1.26neu2}\\
 Here we assume that \(\,p_M=p_N\,\,:\,\, M=N\,\to\,S^1\,\). The coincidence questions concerning a pair of fiberwise maps \(\,f_1,f_2\,:\,N\,\to\,N\,\) amount to fixed point questions for \(\,f_1-f_2+\mathrm{id}\).

In turn, let us study the fixed points of a fiberwise selfmap \(f\) on \(N\) or, equivalently, the coincidences of \(\,(f-\mathrm{id},f_0)\,\). If we fix a fiber of \(N\) and identify its fundamental group with \(\mathbb{Z}^n\), the restriction of \(f\) induces a homomorphism \(\,f|_\ast\,:\,\mathbb{Z}^n\,\to\,\mathbb{Z}^n\,\) (cf. \ref{1.12}) (which extends to a linear selfmap of \(\mathbb{R}^n\)). We have to apply theorem \ref{1.18} to \(\,\overline{L}:=f|_\ast -\mathrm{id}\,\) and to the residue class \(\,[\overline{v}]\in\mathbb{Z}^n\diagup\left(\overline{A}_N-\mathrm{id}\right)\mathbb{Z}^n\,\) determined by the section \(\,f\circ s_{0M}\,\) (cf. \ref{1.15}). Since \(m=n\) and \(\overline{A}_M=\overline{A}_N=:\overline{A}\) the case distinctions in \ref{1.18} can be expressed also in terms of the eigenspace \(W_{+1}\) (for the eigenvalue \(+1\)) of \(\,f|_\ast\,:\,\mathbb{R}^n\,\to\,\mathbb{R}^n\): the cases are numbered by the dimension of \(W_{+1}\) (which equals the codimension of \(\overline{L}\left(\mathbb{R}^n\right)\) in \(\mathbb{R}^n\)); in case 1 the subcases are distinguished by the determinant of the restricted endomorphism \(\,\overline{A}|\,=\,a\cdot \mathrm{id}\,\) on the line \(W_{+1}\) of eigenvectors.

However, when calculating orbit numbers - and hence minimum numbers of (pathcomponents of) fixed points - we should rather focus our attention on the \textbf{\textit{image}} of \(\overline{L}\) and its complement.

For an illustration assume \(m=n=2\) and consider case 1 which is most interesting. Here it is very appropriate to use the coordinate system of \(\mathbb{Z}^2\) (and \(\mathbb{R}^2\)) determined by a choice of integer basis vectors \(\overline{y}_1,\overline{y}_2\in\mathbb{Z}^n\) such that \(\overline{y}_1\) generates the intersection of \(\mathbb{Z}^2\) with the line \(\,\overline{L}\left(\mathbb{R}^2\right)\,=\,\left(f|_\ast-\mathrm{id}\right)\left(\mathbb{R}^2\right)\,\), and \(\overline{y}_2\) realizes the minimal strictly positive distance from this line. Then the gluing matrix \(\overline{A}\), a shift vector \(\overline{v}\) and a generator of the group \(\,\overline{L}\left(\mathbb{Z}^2)\,=\,\left(f|_\ast-\mathrm{id}\right)\left(\mathbb{Z}^2\right)\right)\,\), resp., take the form \[
\overline{A}\,=\,a\cdot \begin{pmatrix} \det\,\overline{A}&\ast \\ 0&1\end{pmatrix},\quad
 \overline{v}=\begin{pmatrix}\overline{v}_1\\\overline{v}_2\end{pmatrix}\quad\text{ and }\quad
\begin{pmatrix}\pm q\\0\end{pmatrix},\,\text{ resp.,}\]
where \(\,a,\) and \(\ \det\,\overline{A}\,\) are equal to \(\,+1\,\) or \(\,-1\,\) and \(\,q\,\) denotes the order of the torsion subgroup \(\,\,\left(\mathbb{Z}^2\cap\overline{L}\left(\mathbb{R}^2\right)\right)\diagup \overline{L}\left(\mathbb{Z}^2\right)\,\,\) of \(\,G\,=\,\mathbb{Z}^2\diagup\overline{L}\left(\mathbb{Z}^2\right)\).

If in addition \(a=-1\) and \(\overline{v}_2 \equiv 0(2)\) we can (and will) choose a representative \(\overline{v}\) of \(\,[\overline{v}]\in\mathbb{Z}^2\diagup \left(\overline{A}_N-\mathrm{id}\right)\mathbb{Z}^2\,\) such that \(\,\overline{v}_2=0\). Then \(\,\beta_{\overline{v}}|\) (cf. theorem \ref{1.18}) is the affine selfmap on \(\mathbb{Z}_q\) defined by\[
\beta_{\overline{v}}|\left([k]\right)\,=\,\left[-(\det\,\overline{A})\cdot(k-\overline{v}_1)\right];\]
thus \(\,\beta_{\overline{v}}|\,\) is an involution (whose only odd order orbits consist of fixed points) or a translation (all of whose orbits have order \(\,q/\mathrm{gcd}(q,\overline{v}_1)\)). Therefore theorem \ref{1.18} implies the following as an easy special case.
\begin{theorem}\label{1.27neu2}
 Let the torus bundle \(N\) over \(S^1\) be determined by the gluing matrix \(\,\overline{A}\in\mathrm{GL}(2,\mathbb{Z})\). Given a fiberwise map \(f\,:\,N\,\to\,N\), let \(\,f|_\ast\,\) denote the (linear) endomorphism of \(\mathbb{Z}^2\) (and, by extension, of \(\mathbb{R}^2\)) induced on the fundamental group of a single fiber of \(N\). Let \(W_{+1}\) be the (real) eigenspace of \(\,f|_\ast\,\) with eigenvalue \(+1\). Recall that \(\,\mathrm{MC}_B(f,\mathrm{id})\,\) is the (fiberwise version of the) classical minimum number of fixed points of the map \(f\) within its (fiberwise) homotopy class. Similarly, \(\,\mathrm{MCC}_B(f,\mathrm{id})\,\) is the minimum number of pathcomponents of fixed point subspaces of \(N\) within the fiberwise homotopy class of \(f\).\\
\textbf{Case 0: }\textnormal{\textbf{dim}} \(\mathbf{W_{+1}=0}\). Here \(\,\mathrm{MC}_B(f,\mathrm{id})=\infty\). In contrast, \(\,\mathrm{MCC}_B(f,\mathrm{id})\,\) is the (finite) number of all orbits of the affine selfmap \(\,\beta_{\overline{v}}\,\) on \(\,\,\mathbb{Z}^2\diagup\!\left(f|_\ast-\mathrm{id}\right)\mathbb{Z}^2\,\) (cf. \ref{1.17}).\\
\textbf{Case 1: }\textnormal{\textbf{dim}}\(\mathbf{W_{+1}=1}\). Here \[
\mathrm{MC}_B(f,\mathrm{id})\,=\,
\begin{cases}
 0 &\text{if }\, \mathrm{MCC}_B(f,\mathrm{id})=0;\\
\infty &\text{else}.
\end{cases}
\]
If \(\overline{A}\) restricts to the identity map \(\mathrm{id}\) on the eigenline \(W_{+1}\), then\[
\mathrm{MCC}_B(f,\mathrm{id})\,=\,q\cdot \left|\overline{v}_2\right|.\]
If \(\,\overline{A}|W_+=-\mathrm{id}\,\) and \(\overline{v}_2\) is odd, then \[
\mathrm{MC}_B(f,\mathrm{id})=\mathrm{MCC}_B(f,\mathrm{id})=0.\]
If \(\,\overline{A}|W_+=-\mathrm{id}\,\) and \(\overline{v}_2=0\), then the value of \(\,\mathrm{MCC}_B(f,\mathrm{id})\,\) depends on \(q,\, \overline{v}_1\) and the determinant of \(\overline{A}\) as follows.\[
\mathrm{MCC}_B(f,\mathrm{id})\,=
\begin{cases}
 r:=\mathrm{gcd}(q,\overline{v}_1) & \text{if }\,\det\,\overline{A}=-1\,\text{ and }\, \frac{q}{r}\,\text{ is odd};\\
2 & \text{if }\,\det\,\overline{A}=+1\,\text{ and }\, q\equiv\overline{v}_1\equiv0(2);\\
1 & \text{if }\,\det\,\overline{A}=+1\,\text{ and }\, q\,\text{ is odd};\\
0 &\text{else}.
\end{cases}
\]
\textbf{Case 2: }\textnormal{\textbf{dim}}\(\mathbf{W_{+1}=2}\). (i.e. \(f|_\ast\equiv\mathrm{id}\)). Here\[
\mathrm{MC}_B(f,\mathrm{id})\,=\,\mathrm{MCC}_B(f,\mathrm{id})\,=\,0.\]
\end{theorem}

In particular, \(\mathrm{MC}_B(f,\mathrm{id})\) can take only the values \(0\) or \(\infty\) whereas \(\mathrm{MCC}_B(f,\mathrm{id})\) or, equivalently, the Nielsen number \(\mathrm{N}_B(f,\mathrm{id})\) seems to capture the fixed point behavior of \(f\) (up to homotopy) very well.
\begin{corollary}\label{1.28neu2}
 There is a fiberwise homotopy from \(f\) to a fixed point free map if and only if \(\,f|_\ast\equiv\mathrm{id}\,\) or else the eigenspace \(W_{+1}\) of \(f|_\ast\) has dimension \(1\) and one of the following conditions hold:
\begin{enumerate}[(i)]
 \item \(\overline{A}|W_{+1}\,=\,\mathrm{id}\) and \(\overline{v}_2=0\);
\item \(\overline{A}|W_{+1}\,=\,-\mathrm{id}\) and \(\overline{v}_2\) is odd;
\item \(\overline{A}|W_{+1}\,=\,-\mathrm{id},\ \overline{v}=(\overline{v}_1,0),\ \det\,\overline{A}=1\) and \(q\equiv0\not\equiv\overline{v}_1(2)\);
\item \(\overline{A}|W_{+1}\,=\,-\mathrm{id},\ \overline{v}=(\overline{v}_1,0),\ \det\,\overline{A}=-1\) and \(q\) is an even multiple of the greatest common denominator \(r\) of \(q\) and \(\overline{v}_1\).
\end{enumerate}
\end{corollary}
This should be compared to the work of D. Gon\c{c}alves, D. Penteado and J. Vieira (cf. [GPV]).
\end{example}
\begin{notations}\label{1.27}\\
 We call a map \(f: M\to N\) \textbf{\textit{fiberwise}} if \(\,p_N\circ f=p_M\) (cf. \ref{1.2}). This agrees with the use of the term ''\textbf{\textit{fiber-preserving}}`` in [GK] (but \textit{not} in [BS] and [Lee], where it is only required that \(p_N\circ f=f'\circ p_M\) for some selfmap \(f'\) of the base \(B\)).

A pair \((f_1,f_2)\) of fiberwise maps is called \textbf{\textit{fiberwise loose}} (or \textbf{\textit{loose over \(\mathit{\mathbf{B}}\)}} or \(\mathit{\mathbf{B}}\)-\textbf{\textit{loose}}, cf. [GK]) if \(f_1,\,f_2\) can be deformed through fiberwise maps so as become coincidence free.

The group compositions in tori and e.g. in the group of fiberwise maps between linear torus bundles are written as additions.
\[
\xymatrix@=1.5cm{
q_k\,:\, \mathbb{R}^k\,\,\ar[r] &\,\,T^k\,=\,\mathbb{R}^k\diagup\mathbb{Z}^k,\qquad k\geq 0,
}\]
denotes the standard covering (or quotient) map. Elements in Euclidean spaces (and their images in tori, resp.) are written \(\,\overline{x},\,\overline{y},\,\overline{u},\,\overline{v},\,\overline{w},\,\ldots\) (and \(x=q_k(\overline{x})\), \(y,\,u,\,v,\,w,\,\ldots\), resp.). A homomorphism \(L\) between tori determines the linear lifting \(\overline{L}\) between Euclidean spaces (which preserves the integer lattice) and vice versa. Our notation makes no destinction between such a lifting \(\overline{L}:\mathbb{R}^m\,\to\,\mathbb{R}^n\) and its restriction \(\overline{L}:\mathbb{Z}^m\,\to\,\mathbb{Z}^n\).

The tildas in \(\widetilde{\omega},\,\widetilde{g},\,\widetilde{\varphi},\,\ldots\) refer to liftings in the fibration \(\mathrm{pr}\) (cf. \ref{1.4}).\\

All homology groups have coefficients in \(\mathbb{Z}_2\). The fundamental class of a closed manifold \(C\) (and the unit interval, and references to publications, resp.) are denoted by \([C]\) (and \(I=[0,1]\), and e.g. [BGZ], resp.). Otherwise square brackets stand for equivalence classes (w.r. to an equivalence relation which should be obvious from the context), e.g. \([\,\overline{x}]=q_k(\overline{x}) =x\) and \([x,t]=\text{equivalence class of the pair } (x,t)\).\\

\(\# S\) (and \(\mathrm{id},\) and \(\det\), resp.) denote the (integer or infinite) number of elements in a set \(S\) (and the identity map, and the determinant at hand, resp.).
\end{notations}


\section{Coincidences in tori}\label{sec2}
In this section we analyze the case \(B=\{\mathrm{point}\}\) (where it is customary to drop the subscript \(B\) from the notations).

Using the group addition in tori we can simplify our exposition significantly. Given continuous maps \(f_1,f_2: T^m\to T^n\), note that the pairs \((f_1,f_2)\) and \(\left(f:=f_1-f_2,\,0\equiv f_2-f_2\right)\) have the same Nielsen, Reidemeister and minimum numbers. Moreover these invariants and further coincidence data are compatible with continuous deformations (see [Ko2]). E.g. we may ''\textit{straighten}`` \(f\) by the canonical homotopy from \(f\) to \(f_{L,f(0)}:=L+f(0)\) (where the Lie group homomorphism \(L\) is extracted from the induced homomorphism
\begin{equation}\label{2.1}
\xymatrix{
 \overline{L}\,\,:\,\,\mathbb{Z}^m=\pi_1\left(T^m\right)\ar[r]^-{f_\ast} & \pi_1\left(T^n\right)=\mathbb{Z}^n
}
\end{equation}
as in \ref{1.12}. In order to obtain the canonical homotopy pick two liftings \(\overline{f},\overline{f}_{L,f(0)} :\mathbb{R}^m\to\mathbb{R}^n\) such that \(\overline{f}(0)=\overline{f}_{L,f(0)}(0)\) and project the affine deformation \(t\,\overline{f}+(1-t)\,\overline{f}_{L,f(0)}\) back to the torus \(T^n\)).

Thus it suffices to consider only pairs of maps of the form \((f,f_0\equiv 0)\) where \(f\) has been straightened.

First we want to understand the pathspace
\begin{equation}\label{2.2}
 E(f,0)= \left\{\left(x,\theta\right)\in T^m\times PT^n\,|\, \theta(0)=f(x),\, \theta(1)=0\right\}
\end{equation}
and its projection \(\mathrm{pr}\) to \(T^m\) (cf. \ref{1.4}).
\begin{theorem}\label{2.3}
 Consider the map\[
f=L+v: T^m\to T^n,\quad m,n\geq 1,\]
where \(L\) is a homomorphism of tori (with linear lifting \(\overline{L}: \left(\mathbb{R}^m,\mathbb{Z}^m\right)\to\left(\mathbb{R}^n,\mathbb{Z}^n\right)\)) and \(v\in T^n\). Then
\begin{enumerate}[(i)]
 \item There is a homotopy equivalence\[
e\circ j: \coprod_{\left[\overline{u}\right]\in\mathbb{Z}^n\diagup\overline{L}\left(\mathbb{Z}^m\right)} \mathrm{ker}\,\overline{L}\diagup\left(\mathbb{Z}^m\cap\mathrm{ker}\,\overline{L}\right)\to E(f,0).\]
\item Assume that \(v\in L\left(T^m\right)\). Then this homotopy equivalence induces a bijection from the Reidemeister set \(\mathbb{Z}^n\diagup\overline{L}\left(\mathbb{Z}^m\right)\) onto the set \(\pi_0\left(E(f,0)\right)\) of pathcomponents of \(E(f,0)\). Let \(E'(f,0)\) denote the union of those pathcomponents of \(E(f,0)\) which correspond to elements in \(\left(\overline{L}\left(\mathbb{R}^m\right)\cap\mathbb{Z}^n\right)\diagup \overline{L}\left(\mathbb{Z}^m\right)\). Then the map\[
\widetilde{g}: C(f,0)\to E'(f,0)\]
which sends \(x\) to \(\left(x, \text{constant path at }0\right)\) is a homotopy equivalence. (In fact, \(\widetilde{g}\) is a homeomorphism onto a strong deformation retract of \(E'(f,0)\).) Clearly, when we compose \(\widetilde{g}\) with the projection \(\mathrm{pr}\) (compare \ref{1.4}), we obtain the inclusion of the coincidence manifold \(C(f,0)\) into \(T^m\).
\end{enumerate}
\end{theorem}
\begin{proof}
 Call two elements \(\left(\overline{x},\overline{u}\right)\) and \(\left(\overline{x}',\overline{u}'\right)\) of \(\mathbb{R}^m\times\mathbb{Z}^n\) \textit{equivalent} if \(\left(\overline{x}',\overline{u}'\right)= \left(\overline{x}+\overline{y},\overline{u}+\overline{L}\left(\overline{y}\right)\right)\) for some \(\overline{y}\in\mathbb{Z}^m\); let \(D_L\) denote the resulting quotient set.

Then there are homotopy equivalences 
\begin{equation}\label{2.4}
\xymatrix{
 D_L:=\left(\mathbb{R}^m\times\mathbb{Z}^n\right)\diagup\!\sim\ar@<0.5ex>[r]^-{e}\quad& \quad E(f,0) \ar@<0.5ex>[l]^-{\eta}
}
\end{equation}
which depend only on the choice of an element \(\overline{v}\in\mathbb{R}^n\) satisfying \(q_n\left(\overline{v}\right) = v\ \) (if \(v\in L\left(T^m\right)\) we will always choose \(\overline{v}\) to lie in \(\overline{L}\left(\mathbb{R}^m\right)\)). Here we define 
\begin{equation}\label{2.5}
 e\left(\left[\overline{x},\overline{u}\right]\right):= \left(q_m(\overline{x}),\, q_n \circ\left(\text{straight path from }\, \overline{L}\,\overline{x}+\overline{v}-\overline{u} \text{ to } 0\right)\right)
\end{equation}
(compare \ref{1.27}); in turn\[
\eta\left(q_m(\overline{x}),\theta\right) := \left[\overline{x}, \overline{L}\,\overline{x}+\overline{v}-\overline{\theta}(0)\right]\]
where \(\overline{\theta}\) lifts the path \(\theta\) to \(\mathbb{R}^n\) such that \(\overline{\theta}(1)=0\) (compare \ref{2.2}). Clearly \(\eta\circ e=\text{identity map id}\); moreover \(e\circ\eta \sim \mathrm{id}\) since every path in \(\mathbb{R}^n\) can be deformed linearly into the straight path with the same endpoints.

Next note that the obvious map (induced by the second projection)
\begin{equation}\label{2.6}
 r: D_L \to \mathbb{Z}^n\diagup\overline{L}\left(\mathbb{Z}^m\right),
\end{equation}
together with \(e\) and \(\eta\) (cf. \ref{2.4}), can be used to label he pathcomponents of \(E(f,0)\). Indeed, given \(\overline{u}\in\mathbb{Z}^n\), we have the homeomorphism
\begin{equation}\label{2.7}
\xymatrix{
j_{\overline{u}}: \mathbb{R}^m \diagup \left(\mathbb{Z}^m\cap\mathrm{ker}\,\overline{L}\right) \ar[r]^-{\cong}
 & r^{-1} \left\{\left[\overline{u}\right]\right\}
}
\end{equation}
which maps \(\left[\overline{x}\right]\) to \(\left[\overline{x},\overline{u}\right]\). Composed with the obvious inclusion this yields the homotopy equivalence \[
\xymatrix{
j_{\overline{u}}|: \left(\mathrm{ker}\,\overline{L}\right)\diagup \left(\mathbb{Z}^m\cap\mathrm{ker}\,\overline{L}\right) \ar[r]^-{\sim} & r^{-1} \left\{\left[\overline{u}\right]\right\}
}\]
of pathconnected spaces. If we pick a representative \(\overline{u}\) for every class \(\left[\overline{u}\right] \in \mathbb{Z}^n\diagup\overline{L}\left(\mathbb{Z}^m\right)\) we obtain the homotopy equivalence \(\ \displaystyle j:=\coprod_{\left[\overline{u}\right]}j_{\overline{u}}|\ \) which, when composed with \(e\) (cf. \ref{2.5}), leads to the first claim of our theorem.

For the rest of the proof consider the special case where \(v\) happens to lie in the image of \(L\) and where we choose \(\overline{v}\) in \(\overline{L}\left(\mathbb{R}^m\right)\). Pick \(\overline{x}_0\in\mathbb{R}^m\) such that \(\overline{L}\left(\overline{x}_0\right)=-\overline{v}\). Thus the corresponding point \(x_0:=q_m\left(\overline{x}_0\right)\in T^m\) lies in the coincidence set \(C(f,0)\). The fiber inclusion\[
\mathrm{pr}^{-1}\left(\left\{x_0\right\}\right) = \left\{\left(x_0,\, \text{loops at } 0\right)\right\} \,\,\subset \,\,E(f,0),\]
together with \(r\circ\eta\) (cf. \ref{2.4} and \ref{2.6}) induces the composite map\[
\xymatrix{
\mathbb{Z}^n=\pi_1\left(T^n;0\right)\ar[r] & \pi_0\left(E(f,0)\right) \ar[r]^-{\left(r\circ\eta\right)_\ast} & \mathbb{Z}^n\diagup\overline{L}\left(\mathbb{Z}^m\right)}
\]
which takes \(\overline{u}\in\mathbb{Z}^n\) to the class \(\left[\overline{u}\right]\) in the quotient group. This shows that \(r\circ\eta\) yields the usual identification of \(\pi_0\left(E(f,0)\right)\) with the Reidemeister set \(R\left(f,0;x_0\right)=\mathbb{Z}^n\diagup\overline{L}\left(\mathbb{Z}^m\right)\) (cf. [Ko2], 2.1). It follows from the definition of \(\eta\) that the bijection \(\left(r\circ\eta\right)_\ast\ \) is independent of the choice of \(\overline{x}_0\).

Furthermore since \[
f \circ q_m = q_n \circ\left(\overline{L}+\overline{v}\right): \mathbb{R}^m\to T^n\]
we have \[
C(f,0) = q_m\left(\overline{L}^{-1}\left(\mathbb{Z}^n-\overline{v}\right)\right) =\bigcup_{\overline{u}\in\mathbb{Z}^n\cap\overline{L}\left(\mathbb{R}^m\right)} q_m\left(\overline{L}^{-1} \left(\left\{\overline{u}-\overline{v}\right\}\right)\right).\]
Given elements \(\overline{u},\overline{u}'\in \mathbb{Z}^n\cap\overline{L}\left(\mathbb{R}^m\right)\), their contributions to this union are equal or disjoint according as they agree modulo \(\overline{L}\left(\mathbb{Z}^m\right)\) or not.

Now choose a direct sum decomposition \(\mathbb{R}^m =K\oplus\mathrm{ker}\,\overline{L}\) and use it in order to deform \(\mathbb{R}^m\) along \(K\) onto any given affine subspace of the form \(\overline{L}^{-1} \left(\left\{\overline{u}-\overline{v}\right\}\right)\) where \(\,\overline{u}\,\in\,\mathbb{Z}^n\cap\overline{L}\left(\mathbb{R}^m\right)\). The resulting deformations are compatible with the translation by any vector \(\overline{y}\in\mathbb{Z}^m\). Moreover\[
j_{\overline{u}+\overline{L}\left(\overline{y}\right)} \left(\left[\overline{x}+\overline{y}\right]\right) = \left[\overline{x}+\overline{y},\overline{u}+\overline{L}\,\overline{y}\right]= \left[\overline{x},\overline{u}\right]= j_{\overline{u}}\left(\left[\overline{x}\right]\right)\]
for all \(\overline{x}\in\mathbb{R}^m\) (compare \ref{2.4} and \ref{2.7}). It follows that \[
q_m\left(\overline{L}^{-1} \left(\left\{\overline{u}-\overline{v}\right\}\right)\right)\quad\cong\quad \overline{L}^{-1} \left(\left\{\overline{u}-\overline{v}\right\}\right)\diagup \left(\mathbb{Z}^m\cap\mathrm{ker}\,\overline{L}\right)\]
gets mapped homomorphically, via \(\ j_{\overline{u}}\ \) (and \(\ e\circ j_{\overline{u}}\ \), resp.), onto a strong deformation retract of the pathcomponent, labelled by \(\left[\overline{u}\right]\), of \(D_L\) (and of \(E(f,0)\), resp.). The deformation is independent of the representative \(\overline{u}\) of the class \(\left[\overline{u}\right]\in \mathbb{Z}^n\diagup\overline{L}\left(\mathbb{Z}^m\right)\); it depends only on the choice of the complement \(K\) of \(\mathrm{ker}\,\overline{L}\) in \(\mathbb{R}^m\). Clearly \(e\circ j_{\overline{u}}\left(\left[\overline{x}\right]\right) =\widetilde{g} \left(q_m(\overline{x}\right)\) for all \(\overline{x}\in\overline{L}^{-1} \left(\left\{\overline{u}-\overline{v}\right\}\right)\) (cf. \ref{2.5} and \ref{2.7}). This completes the proof of theorem \ref{2.3}.
\end{proof}
The next result, together with the discussion of \ref{2.2}, proves theorem \ref{1.13} of the introduction in the case \(B=\{\text{point}\}\).
\begin{corollary}\label{2.8}
 We may choose integer vectors \(\,\overline{u}_1,\ldots,\overline{u}_n\in\mathbb{Z}^n\) which generate \(\overline{L}\left(\mathbb{Z}^m\right)\). Then\[
\mathrm{MCC}(f,0)=\mathrm{N}(f,0)=\left|\det\left(\overline{u}_1,\ldots,\overline{u}_n\right)\right|;\]
this number vanishes (or, equivalently, the pair \((f,0)\) is ''loose``, cf. [Ko2], definition 1.1) if and only if \(\,\overline{L}: \mathbb{R}^m\to\mathbb{R}^n\) is not surjective.

Moreover\[
\mathrm{MC}(f,0)=
\begin{cases}
 \mathrm{N}(f,0)&\text{if }\, \mathrm{N}(f,0)=0 \,\text{ or }\, m=n;\\
\infty &\text{else};
\end{cases}
\]
and \[
\#\mathrm{R}(f,0)=\#\left(\mathbb{Z}^n\diagup\overline{L}\left(\mathbb{Z}^m\right)\right) =
\begin{cases}
 \mathrm{N}(f,0)&\text{if }\, \mathrm{N}(f,0)\neq 0;\\
\infty &\text{if }\, \mathrm{N}(f,0)= 0.
\end{cases}
\]
\end{corollary}
\begin{proof}
 If \(\overline{L}\left(\mathbb{R}^m\right)\neq \mathbb{R}^n\) we may move \(v\) slightly away from \(L\left(T^m\right)\) so that \(C(f,0)=\left(L+v\right)^{-1}(0)\) is empty. Then the pair \((f,0)\) is loose and its Nielsen and minimum numbers vanish. Moreover \(\mathbb{Z}^n\diagup\overline{L}\left(\mathbb{Z}^m\right)\) contains an infinite factor (compare \ref{2.3}(ii)).

Thus assume that the rank of \(\overline{L}\) is \(n\). Then the map \((f,0):T^m\to T^n\times T^n\) is smooth and transverse to the diagonal \(\triangle\) as required in \ref{1.3} and \ref{1.4} (see also [Ko2]). Moreover, since the coincidence datum\[
\widetilde{g}: C(f,0)\to E'(f,0) =E(f,0)\]
(cf. \ref{1.4}) is a homotopy equivalence (cf. \ref{2.3}) each element in the geometric Reidemeister set \(\pi_0\left(E(f,0)\right)\) corresponds to an essential and pathconnected part of the coincidence manifold \(C(f,0)=f^{-1}(0)\). (Clearly, essentiality is detected here by \((m-n)\)-dimensional \(\mathrm{mod\, 2}\) homology.) Thus the pair \((f,0)\) realizes - within its homotopy class - the minimum number of coincidence components (and, if \(m=n\), also of coincidence points) and this agrees with the Nielsen and Reidemeister numbers.

Suppose \(\mathrm{MC}(f,0)\) is finite, i.e. there exists a homotopy from \((f,0)\) to some pair of maps \((f',f'')\) which have only finitely many coincidence points. According to ([Ko2], 3.2, 3.3 and the discussion following \ref{4.4}) such a homotopy induces an isomorphism between the framed bordism groups (and the homology groups with coefficients in \(\mathbb{Z}_2\), resp.) of \(E(f,0)\) and of \(\ E(f',f'')\) which preserve the \(\widetilde{\omega}\)-invariants (and \(\mu_2\left(\widetilde{\omega}\right)\), cf. \ref{1.20}, resp.). This cannot happen if \(m>n\). Indeed, each component of \(C(f,0)\) is an \((m-n)\)-dimensional torus and in view of theorem \ref{2.3}(ii) above we have\[
\mu_2\left(\widetilde{\omega}(f,0)\right)=(1,\ldots,1) \in H_{m-n} \left(E(f,0);\mathbb{Z}_2\right) = \bigoplus_{\pi_0\left(E(f,0)\right)}\mathbb{Z}_2.\]
In contrast \(\widetilde{\omega}(f',f'')\) can be represented by ``small'' generic coincidence manifolds near the finite subset \(C(f',f'')\) of \(T^m\), together with ``small'' (i.e. nearly constant) maps; thus \(\mu_2\left(\widetilde{\omega}(f',f'')\right)=0.\)

Since \(\overline{L}\left(\mathbb{Z}^m\right)\) is a subgroup of \(\mathbb{Z}^n\) there exists a system \(\left\{\overline{u}_1,\ldots,\overline{u}_n\right\}\) of generators of \(\overline{L}\left(\mathbb{Z}^m\right)\). The resulting \(n\times n\)-matrix \(\overline{U}\) has a trivial determinant if \(\mathrm{rank}\,\overline{L}<n\). Otherwise \(\left|\det\left(\overline{U}\right)\right|\) counts the points \(\overline{u}\in\mathbb{Z}^n\) which lie in the paralleliped \(\left\{\overline{u}=\sum t_i\,\overline{u}_i\,|\, 0\leq t_i< 1,\,i=1,\ldots, n\right\}\); but these points form a system of representatives of the classes \(\left[\overline{u}\right]\in\mathbb{Z}^n\diagup\overline{L}\left(\mathbb{Z}^m\right)\) (compare also \ref{1.13}(ii)).
\end{proof}

\begin{example}[\( m=n=2 \)]\label{2.9}
For all selfmaps \(\ f\ \) of the \(2\)-dimensional torus the fixed points of \(\ f\ \) are the coincidence points of the pair \(\ (f-\mathrm{id},0)\). 

Therefore
\[
 \mathrm{MF}(f)\quad=\quad\mathrm{MC}(f,\mathrm{id})\quad=\quad \mathrm{N}(f,\mathrm{id})\quad =\quad\left|\det\left(f_{\ast}-\mathrm{id}\right)\right|.
\]
This is one of the earliest results of topological fixed point theory (cf. [Brou], top of p. 95).
 
\end{example}


\section{Straightening fiberwise maps in linear torus bundles}\label{sec3}
For the remainder of this paper let \(M\) and \(N\) be linear torus bundles with fiber dimensions \(m\) and \(n\), resp., over a smooth closed connected manifold \(B\) of dimension \(b\). Any fiberwise map \(f:M\to N\) determines the section \(s_f:=f\circ s_{0 M}\) of \(p_N\) (the image of the zero section of \(p_M\)) and the fiberwise constant map \(s_f\circ p_M\).
\begin{proposition}\label{3.1}
 The following operations do not change the Nielsen, Reidemeister or minimal numbers \(\mathrm{MCC}_B\) and \(\mathrm{MC}_B\).
\begin{enumerate}[(i)]
 \item Replacing a pair \((f_1,f_2)\) of fiberwise maps from \(M\) to \(N\) by the pair which consists of \(\,f:=f_1-f_2\,\) and \(\,f_0:=s_{0 N}\circ p_M\).
\item Replacing \(f\) by a map \(f'\) which is fiberwise homotopic to \(f\). This can be done e.g. by ``straightening'' \(f\) fiberwise (cf. section \ref{sec2}) while keeping \(s_f=f\circ s_{0 M}\) unchanged; when restricted to any fiber, then \(f'-s_f\circ p_M\) is a group homomorphism of tori. Also any deformation of the section \(s_f\) determines a fiberwise homotopy of straightened maps.
\item Composing \(f\) with isomorphisms of linear torus bundles.
\end{enumerate}
\end{proposition}
It follows that we have to prove theorems \ref{1.13}, \ref{1.18},\ref{1.22neu} and \ref{1.21} only for pairs of the form \((f,f_0)\) where \(f\) is already ``straightened''.
\begin{proof}
 Each of these operations induces a fiberwise homotopy equivalence (or even a homeomorphism) between the pathspaces \(E_B(\cdot,\cdot)\) of the pairs in question as well as an isomorphism between their normal bordism groups which preserves the \(\widetilde{\omega}\)-invariants (compare sections 3 and 4 in [Ko2]). Note that here the coefficient bundle \(\varphi\) (cf. \ref{1.6}) is a pullback of a virtual vector bundle over \(B\) which depends only on \(M\) and \(N\) (and not on \(f_1\)).
\end{proof}
\begin{proposition}\label{3.2}
 Let \(f:M\to N\) be a fiberwise map between \textnormal{trivial} linear torus bundles and assume that the section \(s_f=f\circ s_{0 M}\) is fiberwise homotopic to the zero section \(s_{0 N}\) of \(N\). Let \(f|, 0:T^m\to T^n\) denote the restrictions of \(f,f_0\) to the fibers over some basepoint \(\ast\) of \(B\). 

Then the pairs \((f,f_0)\) and \((f|,0)\) have the same Nielsen, Reidemeister and minimum numbers \(\mathrm{MCC}_B\) (which are therefore known by corollary \ref{2.8}). Moreover we have\[
\mathrm{MC}_B(f,f_0)=
\begin{cases}
 \mathrm{N}_B(f,f_0)&\text{if }\,\mathrm{N}_B(f,f_0)=0 \,\text{ or }\, b=m-n=0;\\
\infty&\text{else}.
\end{cases}
\]

If \(\,\overline{L}\left(\mathbb{R}^m\right)\,=\,\mathbb{R}^n\,\) (cf. \ref{1.12}) and \(\,Q\in\,\pi_0\left(E_B(f,f_0)\right)\,\), then the \(\mathrm{mod\,2}\) homology class \(\,\widetilde{g}_\ast\left(\left[C_Q\right]\right)\,\) (cf. \ref{1.8} and \ref{1.21}) does not vanish and hence \(Q\) is essential.
\end{proposition}
\begin{proof}
 In view of proposition \ref{3.1} we need to consider only the case where \[
f=L\times \mathrm{id}: M=T^m\times B \,\to \,N=T^n\times B\]
and \(L\) is a Lie group homomorphism. Then each pathcomponent of \(E_B(f,f_0)\) contains an element of the form \(\left(x_0,\theta\right)\) where \(\theta\) is a closed loop at \(x_0=0\) in the fiber \(T^n\) over \(\ast\in B\). This inclusion induces a bijection from the orbit set of the Reidemeister operation of \(\pi_1(M)=\pi_1(B)\oplus\mathbb{Z}^m\) on \(\pi_1\left(T^n\right)=\mathbb{Z}^n\) onto the geometric Reidemeister set \(\pi_0\left(E_B(f,f_0)\right)\). Since \(s_f\equiv 0\) the factor \(\pi_1(B)\) acts trivially on \(\mathbb{Z}^n\) and we can identify the Reidemeister sets both of \((f,f_0)\) and of \((f|,0)\) with the quotient group \(\mathbb{Z}^n\diagup\overline{L}\left(\mathbb{Z}^m\right)\).

If the linear lifting \(\overline{L}:\mathbb{R}^m\to\mathbb{R}^n\) of \(L\) is surjective we have only essential Reidemeister classes and each Nielsen class \(C_Q\) is connected (being the product of an affine subtorus of \(T^m\) with \(B\) or \(\{\ast\}\), cf. corollary \ref{2.8} and its proof). Thus \[
\mathrm{MCC}_B=\mathrm{N}_B=\#\mathrm{R}_B=\#\left(\mathbb{Z}\diagup\overline{L}\left(\mathbb{Z}^m\right)\right)
\]
both for \((f,f_0)\) and \((f|,0)\). Moreover \(\mathrm{MC}_B(f,f_0)\) is infinite whenever \(\mathrm{N}(f,f_0)\neq 0\) and \(B\neq \{\ast\}\) since then each Nielsen class must project \textit{onto} \(B\).

If \(\overline{L}\) (and hence \(L\)) is not surjective then \(f\) can be pushed away from \(f_0(M)=\{0\}\times B\) and\[
\mathrm{MC}_B(f,f_0)=\mathrm{MCC}_B(f,f_0)=\mathrm{N}_B(f,f_0) =0\]
but \(\mathrm{R}_B(f,f_0)=\mathbb{Z}^n\diagup\overline{L}\left(\mathbb{Z}^m\right)\) is infinite.
\end{proof}
Propositions \ref{3.1} and \ref{3.2} imply the claims of theorems \ref{1.13}, \ref{1.22neu} and \ref{1.21} as far as they concern the case \(B=S^b,\,b\geq 2\). Indeed, here the fiber bundles \(M\) and \(N\) must be trivial since their gluing maps from the (connected!) equator \(S^{b-1}\) into the discrete structure group \(\mathrm{GL}\left(k,\mathbb{Z}\right)\), \(k=m\) or \(n\), are constant. Thus any section corresponds to a (nullhomotopic!) map from \(S^b\) to the torus \(T^k\).

Furthermore if \(B=S^1\) and \(\mathrm{dim}\left(\overline{L}\left(\mathbb{R}^m\right)\right)\leq n-2\), the claims of theorems \ref{1.13} and \ref{1.18} follow from a simple transversality argument. Indeed, we may deform \(s_f\) into a section of \(p_N\) which does not intersect the (at least \(2\)-codimensional) image of the straightened fiberwise map determined by \(\overline{L}\) and \(s_{0 N}\). This yields a homotopy which moves \(f\) entirely away from \(f_0\). Thus the Nielsen and minimum numbers of the pair \((f,f_0)\) vanish.

The remaining claims of theorems \ref{1.13}, \ref{1.18} and \ref{1.22neu} will be established in sections \ref{sec4} and \ref{sec6} below.\(\hfill\Box\)


\section{Torus bundles over \(S^1\), and related pathspaces}\label{sec4}
In this section we classify fiberwise maps \(f_1,f_2\) between linear torus bundles over the circle and deduce simple descriptions of the homotopy type of the pathspace \(E_B(f_1,f_2)\) which is of central importance in Nielsen theory. Surprisingly, the coincidence map \(\widetilde{g}\) (cf. \ref{1.4}) turns out to yield homotopy equivalences in a significant number of cases.

Given \(m,n\geq 0\), fix invertible matrices \(\overline{A}_M\in\mathrm{GL}\left(m,\mathbb{Z}\right)\) and \(\overline{A}_N\in\mathrm{GL}\left(n,\mathbb{Z}\right)\) with integer coefficients and consider the resulting linear torus bundles \(M\) and \(N\) (cf. \ref{1.14}) over the base manifold \(B=I\diagup 0\sim 1\) (which we identify with the unit circle \(S^1\) via \([t]\leftrightarrow e^{2\pi it}\)). 

For every integer vector \(\overline{v}\in\mathbb{Z}^n\) we define the section
\begin{equation}\label{4.1}
\xymatrix@=1.2cm{
 s_{\overline{v}}\,\,:\,\, S^1 \,\ar[r] &\, N}
\end{equation}
of \(p_N\) (cf. \ref{1.2}) by \[
s_{\overline{v}}\left([t]\right)=\left[q_n\left(t\,\overline{v}\right),t\right],\quad t\in I.\]
\begin{proposition}\label{4.2}
 The assignment\[
\xymatrix@=1.2cm{
\overline{v}\,\; \ar[r] & \;\,s_{\overline{v}}}\]
induces a group isomorphism \(\,\sigma\,\) from  \(\,\,\mathbb{Z}^n\diagup\! \left(\overline{A}_N-\mathrm{id}\right)\left(\mathbb{Z}^n\right)\,\,\)  onto the group of homotopy classes of sections of \(p_N\).
\end{proposition}
\begin{proof}
 First observe that this construction yields a welldefined map. Indeed, given vectors \(\overline{v}\) and \(\overline{w}\) in \(\mathbb{Z}^n\), the homotopy\[
S\left([t],\tau\right):= \left[q_n\left(t\left(\overline{v}+\tau\,\overline{w}\right)+(1-t)\tau\, \overline{A}_N\, \overline{w}\right)\,,\,t\right],\qquad t,\tau\in I,\]
of sections starts from \(s_{\overline{v}}\) (at \(\tau=0\)) and ends at \(s_{\overline{v}'}\) (at \(\tau=1\)) where \(\overline{v}'=\overline{v}-\left(\overline{A}_N -\mathrm{id}\right)\overline{w}\);\(\ \) note that \(S\) is compatible with the identification of \(\left(q_n\left(\tau\,\overline{w}\right),1\right)\in T^n\times \{1\}\) with \(\left(A_N\left(q_n\left(\tau\,\overline{w}\right)\right),0\right)\in T^n\times \{0\}\) in \ref{1.14}.

On the other hand, every homotopy class of sections has a representative \(s\) which maps the basepoint \(\ast=[0]=[1]\) of \(S^1\) to \([0,0]\in N\). Pick a lifting \(\overline{s}: I\to \mathbb{R}^n\) such that \(s\left([t]\right)=\left[q_n\left(\overline{s}(t)\right),t\right]\) and put
\begin{equation}\label{4.3}
 \overline{v}(s):= \overline{s}(1)-\overline{s}(0)\in \mathbb{Z}^n.
\end{equation}
Then the sections \(s\) and \(s_{\overline{v}(s)}\) are homotopic since \(\overline{s}\) can be deformed in \(\mathbb{R}^n\) into the straight path from \(\overline{s}(0)\) to \(\overline{s}(1)\). Any other base point preserving representative \(s'\) of \([s]\) can be deformed into \(s\) by a (not necessarily basepoint preserving) homotopy of the form\[
S\left([t],\tau\right)=\left[q_n\left(\overline{S}(t,\tau)\right),t\right],\qquad t,\tau\in I,
\]
where \(q_n\left(\overline{S}(0,\tau)\right)=q_n\left(\overline{A}_N\left(\overline{S}(1,\tau)\right)\right)\) (compare \ref{4.1} and \ref{1.14}). If we put \(\overline{w}=\overline{S}(1,1)-\overline{S}(1,0)\) we conclude that\[
\overline{v}\left(s'\right)-\overline{v}\left(s\right) = \overline{w}-\overline{A}_N\,\overline{w}.\]
Therefore, the map \([s]\to \left[\overline{v}(s)\right]\) is a welldefined inverse of \(\sigma\).

Clearly our bijection is compatible with the fiberwise addition of sections.
\end{proof}
Next consider a vector \(\overline{v}\in\mathbb{Z}^n\) and a linear map \[
\overline{L}: \left(\mathbb{R}^m,\mathbb{Z}^m\right)\to\left(\mathbb{R}^n,\mathbb{Z}^n\right)\]
such that \(\overline{L}\circ\overline{A}_M=\overline{A}_N\circ\overline{L}\). The induced homomorphism of tori\[
L: T^m\to T^n\]
commutes with the gluing maps of \(M\) and \(N\) and we can define a fiberwise map \(f_{L,\overline{v}}\) by
\begin{equation}\label{4.4}
 f_{L,\overline{v}}\left([x,t]\right)=\left[L\,x+q_n\left(t\,\overline{v}\right),t\right]
\end{equation}
(compare \ref{1.14}). E.g. if \(\overline{L}\equiv 0\) and hence \(L\equiv 0\), we see that
\begin{equation}\label{4.5}
 f_{0,\overline{v}} = s_{\overline{v}}\circ p_M
\end{equation}
(compare \ref{1.2}); if also \(\overline{v}=0\), we obtain the nullmap \(f_0:=f_{0,0}=s_{0N}\circ p_M\) as in the discussion of \ref{1.14}.
\begin{proposition}\label{4.6}
 The assignment\[
\xymatrix@=1.2cm{
\left(\overline{L},\overline{v}\right)\,\,\ar[r] &\,\, f_{L,\overline{v}}\,=\,f_{L,0}+s_{\overline{v}}\circ p_M}\]
determines a welldefined isomorphism from the group\[
\left\{\overline{L}: \mathbb{Z}^m\to \mathbb{Z}^n\,\text{linear}\,|\, \overline{L}\circ\overline{A}_M =\overline{A}_N\circ\overline{L}\right\}\,\; \oplus \;\,\left(\mathbb{Z}^n\diagup\!\left(\overline{A}_N -\mathrm{id}\right) \left(\mathbb{Z}^n\right)\right)\]
onto the group of homotopy classes of fiberwise maps \(f: M\to N\).

The inverse of this isomorphism maps a fiberwise homotopy class \([f]\) to the pair \(\left(\overline{L}=f|_\ast\,,\,\, \sigma^{-1}\left(\left[s_f=f\circ s_M\right]\right)\right)\)
(cf. \ref{1.12}, \ref{4.2} and section \ref{sec3}).
\end{proposition}
\begin{proof}
 Here we have chosen the base point \(\ast=[1]\) of \(B=S^1\) and used the identifications \(T^m=T^m\times \{1\}=F_M\) and \(F_N=T^n\times \{1\}=T^n\) when restricting \(f\) to these standard fibers (cf. \ref{1.12}).

We obtain our result by straightening \(f\) both along each fiber and along the zero section \(s_{0M}\) (see propositions \ref{3.1}(ii) and \ref{4.2}).
\end{proof}
\begin{example}[\( m=n=1 \)](compare [GK])\label{4.7}\\
 Here \(\overline{A}_M,\overline{A}_N\in \mathrm{GL}\left(1,\mathbb{Z}\right)=\{\pm 1\}\). Thus \(M\) equals the 2-dimensional torus \(T\) or the Klein bottle (fibered in the standard way over \(B=S^1\)), and so does \(N\).

A fiberwise map \(f\) is characterized (up to fiberwise homotopy) by two numbers. The first one is an integer which corresponds to the endomorphism \(\overline{L}\) of \(\mathbb{Z}\) and must vanish if \(M\neq N\) (since then \(\overline{L}=-\overline{L}\), cf. \ref{1.16}); it equals the mapping degree of the restriction of \(f\) to a single fiber. The second characterizing number lies in \(\mathbb{Z}\) (or \(\mathbb{Z}_2\), resp.) of \(N=T\) (or \(N=K\), resp.); it measures ``how often \(f\) winds the zero section of \(M\) around the fibers of \(N\)'' (provided \(f\) preserves basepoints in a certain sense).

It can be shown that these two characterizing numbers (and hence the fiberwise homotopy class of \(f\)) are fully determined by the coincidence invariant \(\omega_B(f,f_0)\) (cf. theorem 1.4 in [GK] and \ref{1.22}).\(\hfill\Box\)
\end{example}
Now let us check to what extend the data in proposition \ref{4.6} influence the formula \ref{1.17} which describes the Reidemeister invariant (cf. definition \ref{1.26neu3}) of the pair \((f_{L,\overline{v}},\,f_0)\).
\begin{lemma}\label{4.8}
 Let \(H\subset\mathbb{Z}^n\) be a subgroup such that \(\overline{A}_N(H)=H\). If two integer vectors \(\overline{v},\,\overline{v}'\in\mathbb{Z}^n\) differ by a vector in \(\,H\,+\,\left( \overline{A}_N -\mathrm{id}\right)\!\left(\mathbb{Z}^n\right)\,\) then the selfmaps \(\beta_{\overline{v}}\) and \(\beta_{\overline{v}'}\) (defined as in \ref{1.17}) on the quotient group \(\,G=\mathbb{Z}^n\diagup \!H\,\)  determine the same element in \(\mathcal{R}_B\) (cf. definition \ref{1.20neu} (ii)).
\end{lemma}
\begin{proof}
We need to consider only the case where \(\,\overline{v}'\,=\,\overline{v}+\left(\overline{A}_N -\mathrm{id}\right)\!(\overline{w})\,\) for some \(\overline{w}\in\mathbb{Z}^n\). Then for all \(\overline{u}\in\mathbb{Z}^n\)\[
\beta_{\overline{v}'}\left[\overline{u}+\overline{A}_N(\overline{w})\right]\,= \,\beta_{\overline{v}}[\overline{u}]+\left[\overline{A}_N (\overline{w})\right]\]
so that \((G,\beta_{\overline{v}})\) and \((G,\beta_{\overline{v}'})\) are equivalent via the translation by \(\left[\overline{A}_N(\overline{w})\right]\).
\end{proof}
Lemma \ref{4.8} implies that the Reidemeister invariant depends only on the homotopy class (over \(S^1\)) of \(f_{L,\overline{v}}\). In addition it allows sometimes to simplify the calculation of orbit numbers (see e.g. example \ref{4.12neu} below).\\

Next we turn to the target of the map \(\widetilde{g}\) (cf. \ref{1.4}). In view of propositions \ref{3.1} and \ref{4.6} it suffices to focus our attention on spaces of the form \(E_B\left(f_{L,\overline{v}},f_0\right)\) where \(L\) and \(\overline{v}\) are as in \ref{4.4} and \ref{4.6}. These data allow us also to define a selfhomeomorphism \(\overline{b}\) of \(D_L =\left(\mathbb{R}^m\times \mathbb{Z}^n\diagup\!\sim\right)\) (cf. \ref{2.4}) by 
\begin{equation}\label{4.9}
 \overline{b}\left([\overline{x},\overline{u}]\right) := \left[\overline{A}_M(\overline{x}), \overline{A}_N(\overline{u}-\overline{v})\right],\quad \overline{x}\in\mathbb{R}^m,\overline{u}\in\mathbb{Z}^n;
\end{equation}
(it induces the selfmap \(\beta=\beta_{\overline{v}}\,\) on \(\,\mathbb{Z}^n\diagup\overline{L}\left(\mathbb{Z}^m\right)\) which takes \([\overline{u}]\) to \(\left[\overline{A}_N\left(\overline{u}-\overline{v}\right)\right]\), cf. \ref{1.17}).
\begin{theorem}\label{4.10}
 \begin{enumerate}[(i)]
  \item There is a homotopy equivalence \(e\) between the mapping torus\[
D:=\left(D_L\times I\right)\diagup\left([\overline{x},\overline{u}],1\right) \sim\left(\overline{b}\left([\overline{x},\overline{u}]\right),0\right)\]
and \(E_B\left(f_{L,\overline{v}},f_0\right)\) (and \(e\) is compatible with the natural projections to \(\ \ I\diagup\!\sim\ \ =\ \ S^1\ \)).

In particular, \(e\) induces a canonical bijection from the set of all orbits of (the \(\mathbb{Z}\)-action determined by) \(\beta\) in \(\mathbb{Z}^n\diagup\overline{L}\left(\mathbb{Z}^m\right)\) onto the set \(\pi_0\left(E_B\left(f_{L,\overline{v}},f_0\right)\right)\) of pathcomponents, i.e. the Reidemeister set \(R_B\left(f_{L,\overline{v}},f_0\right)\). Such an orbit of \(\beta\) has an odd number of elements if and only if the corresponding pathcomponent \(Q\) of \(E_B\left(f_{L,\overline{v}},f_0\right)\) satisfies the following condition: the homomorphism\[
\left(p_M\circ\mathrm{pr}\right)_\ast : H_1\left(Q;\mathbb{Z}_2\right)\to H_1\left(S^1;\mathbb{Z}_2\right)\]
(induced by the natural projection) is surjective.
\item Assume that \(\overline{v}\in\overline{L}\left(\mathbb{R}^m\right)\). Let \(E_B'\left(f_{L,\overline{v}},f_0\right)\) denote the union of those pathcomponents of \(E_B\left(f_{L,\overline{v}},f_0\right)\) which corresponds to the orbits of \(\beta\) which lie in \(\left(\overline{L}\left(\mathbb{R}^m\right)\cap\mathbb{Z}^n\right)\diagup \overline{L}\left(\mathbb{Z}^m\right)\). 

Then the map\[
\widetilde{g}: C_B\left(f_{L,\overline{v}},f_0\right)\to E_B'\left(f_{L,\overline{v}},f_0\right)\]
which sends \(x\) to \(\left(x,\,\text{constant path at } f_{L,\overline{v}}(x)=f_0(x)\,\text{(compare \ref{1.4})}\right)\) is a homotopy equivalence. (In fact, \(\widetilde{g}\) is a homeomorphism onto a strong deformation retract of \(E_B'\left(f_{L,\overline{v}},f_0\right)\)).
 \end{enumerate}
\end{theorem}
\begin{proof}
 By definition (cf. \ref{4.6}) \(f_{L,\overline{v}}\) is built up fiberwise from the maps \[
\xymatrix@=1,2cm{
f_t=L+q_n(t\,\overline{v})\;:\; T^m\ar[r]\,&\, T^n,\quad t\in I.}\]
Consider the corresponding homotopy equivalences \(e_t\) (cf. \ref{2.4}). They fit together to yield the homotopy equivalence\[
\xymatrix@=1.2cm{
e:\left(D_L\times I\right)\diagup\!\sim \,\,\ar[r]& \,\, E_B\left(f_{L,\overline{v}},f_0\right)}.\]
Indeed, the effects of the gluing maps \(A_M\) and \(A_N\) (cf. \ref{1.14}) transform \(e_1\) into \(e_0\circ \overline{b}\). Moreover, the process of deforming a path in \(\mathbb{R}^n\) linearly into the straight path with the same endpoints is compatible with the identification \[
\overline{A}_N :\mathbb{R}^n\times \{1\} \cong \mathbb{R}^n\times \{0\}\]
(compare the discussion of \ref{2.5}).

Recall from \ref{2.6} that the pathcomponents of \(D_L\) (and hence of \(D_L\times I\)) can be labelled bijectively by the elements \([\overline{u}]\in\mathbb{Z}^n\diagup \overline{L}\left(\mathbb{Z}^m\right)\). In the mapping torus \(\,\,\left(D_L\times I\right)\diagup\!\sim\ \) the top end of the component labelled \([\overline{u}]\) gets glued to the bottom end of the component whose label is \(\beta\left([\overline{u}]\right)\). Therefore each orbit\[
\ldots\,, \,\,\beta^{-1}[\overline{u}]\,, \,\,\overline{u}\, ,\,\,\beta[\overline{u}] \,, \,\,\beta^2[\overline{u}]\,, \,\,\ldots\]
of \(\beta\) (of order \(q\)) corresponds - via \(e\) - to a pathcomponent \(Q\) of \(E_B\left(f_{L,\overline{v}},f_0\right)\) (which, when projected to the base \(S^1\), winds around it \(q\) times).

In order to establish the second claim of theorem \ref{4.10} we have to make the retractions and deformations in the proof of theorem \ref{2.3}(ii) compatible with the gluing in the mapping torus \(\left(D_L\times I\right)\diagup\!\sim\ \). \(\ \)If \(\overline{v}\) and \(\overline{u}\) lie in \(\overline{L}\left(\mathbb{R}^m\right)\cap\mathbb{Z}^n\) then for every \(t\in I\)
\[
q_m\left(\overline{L}^{-1}\left(\left\{\overline{u}-t\,\overline{v}\right\}\right)\right)\quad \cong\quad \overline{L}^{-1}\left(\left\{\overline{u}-t\,\overline{v}\right\}\right)\diagup \left(\mathbb{Z}^m\cap\mathrm{ker}\,\overline{L}\right)
\]
is a strong deformation retract of \[
\mathbb{R}^m\diagup \left(\mathbb{Z}^m\cap\mathrm{ker}\,\overline{L}\right)\; \cong\; r^{-1}\left\{[\overline{u}1]\right\}\subset D_L\]
(cf. \ref{2.7}). The (affine) retraction depends on a choice: it involves (the projection \(p_K\) along \(K\) in) a splitting \(\mathbb{R}^m= \mathrm{ker}\,\overline{L}\oplus K\). This can be isotoped linearly into the splitting \(\mathbb{R}^m= \mathrm{ker}\,\overline{L}\oplus \overline{A}_M(K)\). We can use such an isotopy to make the necessary corrections over a neighbourhood of the gluing parameter \([0]=[1]\) in the base \(B=S^1\).
\end{proof}
As a consequence we can settle Case 0 in theorem \ref{1.18}.
\begin{corollary}\label{4.12}
 Assume that the linear lifting \(\overline{L}: \mathbb{R}^m \to\mathbb{R}^n\) of \(L\) is surjective. Then \(\mathrm{MC}_B\left(f_{L,\overline{v}},f_0\right)=\infty\) and\[
\mathrm{MCC}_B\left(f_{L,\overline{v}},f_0\right) =\mathrm{N}_B\left(f_{L,\overline{v}},f_0\right)=\#\pi_0\left(E_B\left(f_{L,\overline{v}},f_0\right)\right)\]
equals the number of orbits of (the \(\mathbb{Z}\)-action defined by) the selfmap \(\beta=\beta_{\overline{v}}\) on \(\mathbb{Z}^n\diagup\overline{L}\left(\mathbb{Z}^m\right)\) (cf. \ref{1.17}).
\end{corollary}
\begin{proof}
 Here the smooth map\[
\left(f_{L,\overline{v}},f_0\right): M\to N\times_B N\]
is transverse to the diagonal \(\triangle\) (compare [GK], 1.4). The pathcomponents of the resulting \emph{generic} coincidence manifold \(C_B\!\left(f_{L,\overline{v}},f_0\right)\) correspond bijectively - via the coincidence datum \(\widetilde{g}\) (cf. \ref{1.4}) - to the pathcomponents of \(E_B\left(f_{L,\overline{v}},f_0\right)= E_B'\left(f_{L,\overline{v}},f_0\right)\). They are all essential since \(C_B\left(f_{L,\overline{v}},f_0\right)\) is a closed manifold and \(\widetilde{g}\) is a homotopy equivalence (too see this, homology with coefficients in \(\mathbb{Z}_2\) suffices). This shows that the Nielsen number agrees both with the Reidemeister and minimum numbers of \(\left(f_{L,\overline{v}},f_0\right)\). 

Finally recall from corollary \ref{2.8} that the restriction of \(\left(f_{L,\overline{v}},f_0\right)\) to any given fiber \(p_N^{-1}\left\{[t]\right\}\cong T^m,\,[t]\in S^1,\) cannot be deformed to become coincidence free. Thus any pair of maps from \(M\) to \(N\) which is fiberwise homotopic to \(\left(f_{L,\overline{v}},f_0\right)\) must have at least one coincidence point in each fiber. Therefore \(\mathrm{MC}_B\left(f_{L,\overline{v}},f_0\right) =\infty\).
\end{proof}
\begin{example}\label{4.12neu}
 Suppose that \(M=N\) is a fiberwise product of \(n\) Klein bottles over \(S^1\) (i.e. \(\,\overline{A}_M=\overline{A}_N=-\mathrm{id}\); compare example \ref{4.7}). According to proposition \ref{4.6} fiberwise homotopy classes \(\,[f\,:\,M\,\to\,N]\,\) are classified by arbitrary \(n\times n\)-matrices \(\overline{L}\) with integer entries and by residue classes \(\,[\overline{v}]\in\mathbb{Z}^n\diagup 2\mathbb{Z}^n\). Consider the case where \(\overline{L}\) is diagonal with odd entries \(a_{11},\ldots,a_{nn}\). Then the group \(\mathbb{Z}^n\) is fully generated by its subgroups \(\,\overline{L}\left(\mathbb{Z}^m\right)\,\) and \(\,\left(\overline{A}_M-\mathrm{id}\right)\left(\mathbb{Z}^n\right).\) According to lemma \ref{4.8} the selfmap \(\beta_{\overline{v}}\) of \[
G\,:=\, \mathbb{Z}^n\diagup\overline{L}\left(\mathbb{Z}^m\right)\,=\,\bigoplus \mathbb{Z}_{|a_{ii}|}\]
has the same orbit behavior as the involution \(\,\beta_0=\overline{A}_{N\ast}=-\mathrm{id}\,\) on \(G\). Here the only odd order orbit consists of the fixed point \(0\in G\); all orther orbits have order 2. Thus the values taken by the functions \(\nu_{\mathrm{odd}},\,\nu_{\mathrm{even}},\,\nu_{\infty}\) and \(\nu\) on \(\,(G,\beta_{\overline{v}})\,\) are\[
1,\ \frac{1}{2}\,\left(\prod\,\left|a_{ii}\right| -1\right),\ 0\ \text{and}\ \frac{1}{2}\,\left(\prod\left|a_{ii}\right|+1\right),\ \text{resp.}\]
(all of them independent of \(\overline{v}\)).\(\hfill\Box\)
\end{example}

Finally we prove proposition \ref{1.21neu2}. We may assume that \(\overline{L}\neq 0\). First consider the selfmaps \(\eta_+\) and \(\eta_-\) of the finite group \(\,G\,=\,\mathbb{Z}^2\diagup\overline{L}\left(\mathbb{Z}^2\right)\,\) induced by complex multiplication with \(i+1\) and \(i-1\), resp. They have isomorphic cokernels and kernels of order at most \(2\) since \(\eta_-=i\cdot\eta_+\) and \[
(i\pm 1)\cdot \mathbb{Z}^2 \,=\, \left\{\overline{z}=\overline{z}_1+i\,\overline{z}_2\in\mathbb{Z}^2\, \left|\right.\, \overline{z}_1+\overline{z}_2\,\text{even}\right\}.\]

If \(k+l\) (cf. proposition \ref{1.21neu2}) is odd (or, equivalently, \(k^2+l^2\equiv 1\,(4)\)), then \(\mathbb{Z}^2\) is spanned by \(\,\overline{L}\left(\mathbb{Z}^2\right)\,\) and \(\,(i-1)\cdot\mathbb{Z}^2\). Thus \(\eta_\pm\) is onto and hence bijective. Moreover \(\beta_{\overline{v}}\) has the same orbit behavior as \(\,\beta_0=i\cdot\mathrm{id}\,\) (cf. lemma \ref{4.8}).

If \(k+l\) is even (and therefore so is \(k^2+l^2\)), then \(\,k+i\,l\,=\,(i+1)\overline{w}\,\) for some \(\,\overline{w}\in\mathbb{Z}^2\,\) and we have the exact sequence
\begin{equation}\label{4.13neu}
 \xymatrix@=1cm{
0\,\ar[r] & \,\mathbb{Z}_2\,\ar^-{\xi}[r] & \,G\, \ar^-{\eta_+}[r] & \,G\,\ar^-{\zeta}[r] &\, \mathbb{Z}_2 \, \ar[r] &\, 0
}
\end{equation}
with \(\,\xi(1):=[\overline{w}]\,\) and \(\,\zeta[\overline{z}]:=[\overline{z}_1+\overline{z}_2]\).

Recall that \(\,\beta_{\overline{v}}[\overline{z}]=i[\overline{z}]-i[\overline{v}],\;[\overline{z}]\in G\). Thus we see by induction that \[
\beta^s_{\overline{v}}[\overline{z}]=i^s[\overline{z}]-\left(i^s+i^{s-1}+\ldots+i\right)[\overline{v}] \]
for all \(s\geq 1\). Clearly \(\,\beta^4_{\overline{v}}\equiv\mathrm{id}\,\) so that \(\,\beta_{\overline{v}}\,\) can have only orbits of order \(1\), \(2\) and \(4\).

We can characterize the fixed points \(\,[\overline{z}]\,\) of \(\,\beta_{\overline{v}}\,\) by the condition \(\,\eta_+[\overline{z}]=[\overline{v}]\). The number \(\nu_1\) of such points equals \(\#\mathrm{ker}\,\eta_+\,\) except when \(\,\zeta[\overline{v}]=[\overline{v}_1+\overline{v}_2]\neq 0\,\) (cf. \ref{4.13neu}).

If an orbit of order \(2\) exists, it has the form \(\,\mathcal{O}=\left\{[\overline{z}],\beta_{\overline{v}}[\overline{z}]\right\}\,\) where \[
\beta^2_{\overline{v}}[\overline{z}]-[\overline{z}]\, =\,-2[\overline{z}]-(i-1)[\overline{v}]\,=\, \eta_-\,\left(\eta_+[\overline{z}]-[\overline{v}]\right)\,=\,0 \]
but \(\,-i\left(\beta_{\overline{v}}[\overline{z}]-[\overline{z}]\right)\,= \,\eta_+[\overline{z}]-[\overline{v}]\,\neq \,0\). This can happen only if \(\,k\equiv l(2)\,\) and \(\,\eta_+[\overline{z}]-[\overline{v}]\,\) is the unique nontrivial element \([\overline{w}]\) of \(\mathrm{ker}\,\eta_-=\mathrm{ker}\,\eta_+\,\) (cf. \ref{4.13neu}), i.e. \(\,\zeta[\overline{v}+\overline{w}]=0\,\) and \(\,\mathcal{O}=\eta_+^{-1}\left\{[\overline{v}+\overline{w}\right\}\); note that \(\,\zeta[\overline{v}+\overline{w}]=[\overline{v}_1+\overline{v}_2+l]\,\) since \(\,(i+1)\cdot\overline{w}=k+i\,l\).

The remaining orbits have order \(4\) and the cardinalities of all orbits sum up to yield the cardinality of the quotient group \(\,G=\mathbb{Z}^2\diagup\overline{L}\left(\mathbb{Z}^2\right)\). Here every residue class contains a unique integer vector which lies in the halfopen parallelogram \(P\) (cf. \ref{6.6}) spanned by \(\,\overline{L}(1)=k+i\,l\,\) and \(\,\overline{L}(i)=-l+i\,k\). Thus \(\,\# G=\det\left(\overline{L}(1),\overline{L}(i)\right)=k^2+l^2\,\) and the proof of proposition \ref{1.21neu2} is complete.


\section{Computing obstruction groups}\label{sec5}
We continue to discuss the case \(B=S^1\). In this section we develop a technique which allows us often to describe the normal bordism groups in which the \(\omega\)-invariants lie, and the effect of the Hurewicz homomorphisms into the homology with coefficients in \(\mathbb{Z}_2\).

Given invertible matrices \(\overline{A}_M\in \mathrm{GL}\left(m,\mathbb{Z}\right)\) and \(\overline{A}_N\in \mathrm{GL}\left(n,\mathbb{Z}\right)\) as in section \ref{sec4}, consider the manifolds
\begin{equation}\label{5.1}
 \overline{M}:= \mathbb{R}^m\times I\diagup(\overline{x},1)\sim \left(\overline{A}_M(\overline{x}),0\right),\,\overline{x}\in\mathbb{R}^m,
\end{equation}
and
\begin{equation}\label{5.1'}\tag{\ref{5.1}'}
  \overline{N}:= \mathbb{R}^n\times I\diagup(\overline{u},1)\sim \left(\overline{A}_N(\overline{u}),0\right),\, \overline{u}\in\mathbb{R}^n,
\end{equation}
which are total spaces of covering maps over \(M\) and \(N\) (cf. \ref{1.14}) and of vector bundles over \(S^1=I\diagup 0\sim 1\).

The pullbacks \(p_M^\ast\left(\overline{M}\right)\) and \(p_N^\ast\left(\overline{N}\right)\) of these vector bundles are canonically isomorphic to the tangent bundles \(TF(p_M)\) and \(TF(p_N)\) along the fibers of \(p_M\) and \(p_N\), resp., (and hence stably isomorphic to the full tangent bundles \(TM\) and \(TN\), resp.). Thus the virtual coefficient bundle 
\begin{equation}\label{5.2}
 \varphi=p_M^\ast \left(\overline{N}-\overline{M}\right)= p_M^\ast (\lambda_d)\in \widetilde{KO}M)=KO(M)\diagup \{\text{trivial vector bundles}\}
\end{equation}
(cf. \ref{1.6}) is independent of the maps \(f_1,f_2\) and depends only on 
\begin{equation}\label{5.3}
 d:=\det \left(\overline{A}_M\right)\cdot\det\left(\overline{A}_N\right)\in\{\pm 1\}
\end{equation}
here \(\lambda_d:=\mathbb{R}\times I\diagup(r,1)\sim (d\cdot r,0)\) denotes the corresponding line bundle over \(S^1\).

In this section we will identify the torus \(T^n\) with the fiber \(F_M\) of \(p_M\) over the base point \(\ast:=[1]\,\,\in\,\, I\diagup 1\sim 0\,\,=\,\,S^1\) via the homeomorphism \(x \to [x,1],\,x\in T^m\).
\begin{proposition}\label{5.4}
 Given fiberwise maps \(\,f:=f_{L,\overline{v}},\,f_0: M\to N\,\) over \(S^1\) as in \ref{4.4} and \ref{4.5}, put \(E:=E_B(f,f_0)\) for short and let \(E|:=E|F_M\) denote the restriction to the fiber \(F_M=p_M^{-1}\left\{[1]\right\}=T^m\).
\begin{enumerate}[(i)]
 \item There is a commuting diagram of long exact sequences\[
\xymatrix@=1.09cm{
\ldots \ar[r] & \Omega_\ast^{\mathrm{fr}}(E|)\ar[r]^-{b_\ast -d\cdot\mathrm{id}} \ar[d]^-{\mathrm{pr}|_\ast} & \Omega_\ast^{\mathrm{fr}}(E|) \ar[r]^-{\mathrm{incl}_\ast} \ar[d]^-{\mathrm{pr}|_\ast} &
 \Omega_\ast(E; \mathrm{pr}^\ast (\varphi)) \ar[r]^-{\pitchfork} \ar[d]^-{\mathrm{pr}_\ast} & \Omega_{\ast-1}^{\mathrm{fr}}(E|)\ar[r] \ar[d]^-{\mathrm{pr}|_\ast} & \ldots\\
\ldots \ar[r] &\Omega_\ast^{\mathrm{fr}} (T^m) \ar[r]^-{A_{M\ast}-d\cdot\mathrm{id}}&
\Omega_\ast^{\mathrm{fr}} (T^m)\ar[r]^-{\mathrm{incl}_\ast} &
\Omega_\ast(M;\varphi) \ar[r]^-{\pitchfork} &
\Omega_{\ast-1}^{\mathrm{fr}} (T^m)\ar[r] & \ldots,
}
\]
and the homomorphism \(\pitchfork\) maps \(\widetilde{\omega}_B(f,f_0)\) (and \(\omega(f,f_0)\), resp.) to the corresponding \(\omega\)-invariants of the restricted pair \(\left(f|F_M, f_0|F_M\right)\) (compare section \ref{sec2}).

Here \(\,\Omega_\ast^{\mathrm{fr}}\,\) denotes framed bordism; moreover the selfmap \(b\) of \(E|\) is defined by \[b(x,\theta)=\left(A_M(x),A_N\circ\left(c_{-\overline{v}}\ast\theta\right)\right)\]
where \(c_{-\overline{v}}\) denotes the loop at \(\theta(0)\) in \(T^n\) which lifts to a path \(\overline{c}_{-\overline{v}}:\,I\,\to\, \mathbb{R}^n\,\) with constant velocity \(-\overline{v}\) and \(\ast\) means concatenation.
\item Similarly there is the commuting diagram of long exact sequences (in homology with \(\mathbb{Z}_2\)-coefficients)
\[
\xymatrix@=1.1cm{
\ldots \ar[r] & H_\ast (E|)\ar[r]^-{b_\ast -\mathrm{id}} \ar[d]^-{\mathrm{pr}|_\ast} & 
H_\ast(E|) \ar[r]^-{\mathrm{incl}_\ast} \ar[d]^-{\mathrm{pr}|_\ast} & 
H_\ast(E) \ar[r] \ar[d]^-{\mathrm{pr}_\ast} &
H_{\ast-1}(E|)\ar[r] \ar[d]^-{\mathrm{pr}|_\ast} & \ldots\\
\ldots \ar[r] &  H_\ast (T^m) \ar[r]^-{A_{M\ast}-\mathrm{id}}&
H_\ast (T^m)\ar[r]^-{\mathrm{incl}_\ast} &
H_\ast (M) \ar[r] &
H_{\ast-1} (T^m)\ar[r] & \ldots.
}\]
This is related to the diagram in (i) above by two commuting ladders which involve Hurewicz homomorphisms as rungs.
\end{enumerate}
\end{proposition}
\begin{proof}
 Consider any \textit{relative} normal bordism class\[
c=\left[C,g,\overline{g}\right]\in\Omega_\ast\left(M,M-F_M;\varphi\right)\]
i.e. the compact smooth manifold \(C\) has possibly a boundary \(\partial C\), \(g\) maps \(\left(C,\partial C\right)\) to \(\left(M,M-F_M\right)\) and \(\overline{g}\) is a stable trivialization of the vector bundle \(TC\oplus g^\ast(\varphi)\). After a small deformation we may assume that \(p_M\circ g\) is smooth and transverse to \(\{\ast\}:=\left\{[1]\right\}\subset S^1\). Restrict the data of \(c\) to the closed 1-codimensional submanifold \(\left(p_M\circ g\right)^{-1}\{\ast\}=g^{-1}\left(F_M\right)\) which inherits a framing since \(\lambda_d|\{\ast\}\) is trivial. This procedure yields the isomorphism
\begin{equation}\label{5.5}
\xymatrix{
 \pitchfork^{\mathrm{rel}} : \Omega_\ast\left(M,M-F_M;\varphi\right)\ar[r]^-{\cong} & \Omega_{\ast-1}^{\mathrm{fr}}(T^m).
}
\end{equation}
Furthermore for small \(\varepsilon>0\) the obvious composite inclusion map
\begin{equation}\label{5.6}
 in: T^m \approx T^m\times\{1-\varepsilon\}\subset T^m\times(0,1)\approx M-F_M
\end{equation}
(compare \ref{1.14}) is a homotopy equivalence and hence induces a canonical isomorphism
\begin{equation}\label{5.7}
\xymatrix{
 in_\ast: \Omega_\ast^{\mathrm{fr}}(T^m) \ar[r]^-{\cong} & \Omega_\ast^{\mathrm{fr}}\left(M-F_M\right) = \Omega_{\ast}\left(M-F_M; p_M^\ast(\lambda_d)\right|).
}
\end{equation}

If we use these isomorphisms \(\pitchfork^\ast\) and \(in_\ast\) to simplify the exact normal bordism sequence of the pair \(\left(M,M-F_M\right)\) (compare e.g. [CF], p. 13) we obtain the lower sequence of the diagram in \ref{5.4}(i). The second and third homomorphisms are induced by the inclusion map and by transverse intersection with the fiber \(F_M\). It remains to calculate the first homomorphism (which  is derived from the boundary operator \(\partial\)). Given a framed singular manifold\[
\xymatrix{
h: H\to F_M=T^m\times \{1\} \left( \ar[r]^-{A_M}_-{\cong} \right.& \left.T^m\times \{0\}\right)
}\]
(compare \ref{1.14}), take its product with the inclusion map \(i\) from the interval \(J=[1-\varepsilon,1]\cup[0,\varepsilon]\diagup 1\sim 0\) into the basis circle \(B=I\diagup 1 \sim 0\). Then the boundary of \(h\times i\) represents \(\partial\circ \left(\pitchfork^{\mathrm{rel}}\right)^{-1} \left([h]\right)\). When we apply the isomorphism \(in_\ast^{-1}\) the two boundary parts \(h_\varepsilon\) and \(h_{1-\varepsilon}\) get shifted to \(A\circ h\) and \(h\), resp. Therefore
\[
\pm in_\ast^{-1}\circ\partial\circ\left(\pitchfork^{\mathrm{rel}}\right)^{-1}=A_{M\ast} -d\,\mathrm{id};
\]
the factor \(d=\pm 1\) (cf. \ref{5.3}) appears due to the possible switch of framings induced from \(\lambda_d\) at \([1]=[0]\in B\). We obtain the upper exact sequence in the diagram of proposition \ref{5.4}(i) in an analoguous fashion. The compatibility of the \(\omega\)-invariants with \(\pitchfork\) is seen by inspecting their definitions.

For the proof of our claim (ii) consider the isomorphisms
\begin{equation}\label{5.8}
 \xymatrix{
H_{\ast}(T^m)\ar[r]_-{\cong}^-{\times} & H_{\ast+1}\left(T^m\times \left(J, J-\{\ast\}\right)\right)\ar[r]_-{\cong}^-{i_\ast} & H_{\ast+1}\left(M,M-T^m\right)
}
\end{equation}
(compare \ref{5.5}) where \(J\) is again a closed interval around the basepoint \(\ast=[1]\) in \(S^1\), \(\times\) is defined by the homology cross product with the generator of \(H_1\left(J,J-\{\ast\}\right)\cong H_1\left(J,\partial J\right)\) (cf. [S], pp. 234-235) and \(i\) denotes the composite of a fiber preserving homeomorphism over \(J\) and the excision inclusion of \(p_M^{-1}\left(J,J-\{\ast\}\right)\) into \(\left(M,M-T^m\right)\). We compare the two long exact homology sequences of the two space pairs at the right hand side in \ref{5.8}. Due to the compatibilities of the cross product (cf. [S], p. 235, fact 15) the boundary homomorphism in the first of these sequences can be identified with the diagonal map into \(H_\ast(T^m)\oplus H_\ast(T^m)=H_\ast\left(T^m\times \partial J\right)\). Since the two points of \(\partial J\) lie on different sides of \(\ast\) in \(S^1\) the gluing map of \(M\) comes into play again when we use the isomorphism induced by the inclusion map \(in\) (cf. \ref{5.6}). Thus the exact homology sequence of the pair \(\left(M,M-T^m\right)\) takes the form described in the lower line of the diagram in claim (ii). 

The upper sequence in (ii) can be obtained in a similar way. Here it may be helpful to use the mapping torus model of \(E\) (cf. theorem \ref{4.10}(i)) and to exploit the fact that it is the total space of a locally trivial fibration.

It is not hard to check that all the identifications and other homomorphisms involved in our two diagrams are compatible with the \(\mathrm{mod}\,2\)-Hurewicz homomorphisms.
\end{proof}
\begin{remark}\label{5.9}
 If \(\,\overline{L}\left(\mathbb{R}^m\right)=\mathbb{R}^n\,\) then \(\,\mu_2\circ\pitchfork\left(\widetilde{\omega}_{B,Q}(f,f_0)\right)\,\) (cf. \ref{6.14}) is nontrivial for every pathcomponent \(Q\) of \(\,E_B(f,f_0)\,\) (cf. \ref{2.3}ii) and hence so is \(\,\mu_2\left(\widetilde{\omega}_{B,Q}(f,f_0)\right)\,\) by proposition \ref{5.4}.
\end{remark}


\section{The case \(\mathrm{dim} \left(\overline{L}(\mathbb{R}^m)\right)=n-1\)}\label{sec6}
Throughout this section we consider again fiberwise maps \(\,f:=f_{L,\overline{v}},\ f_0\,:\,M\,\to\,N\,\) over \(S^1\) as in \ref{4.4}, \ref{4.5} and proposition \ref{5.4} and we assume that the image of the linear map \(\overline{L}: \mathbb{R}^m\, \to\,\mathbb{R}^n\) has dimension \(n-1\). As usual \(H_\ast (-)\) denotes homology with \textbf{\textit{coefficients in \(\mathit{\mathbf{\mathbb{Z}_2}}\)}}.

According to theorem 2.3(i) the pathcomponents of \(E|\,=E\left(f|T^m, 0\right)\) are labelled by 
\[K:= \mathbb{Z}^n\diagup\overline{L}(\mathbb{Z}^m)\]
and each of them is homotopy equivalent to a torus of dimension \(m-n+1\). Therefore we may identify \(H_{m-n+1} (E|)\) with \(\bigoplus_K \mathbb{Z}_2\). Then proposition 5.4 yields the exact sequence
\begin{equation}\label{6.1}
\xymatrix@=2.0cm{
 \displaystyle\bigoplus_K \,\mathbb{Z}_2\, \ar[r]^-{b_\ast-\mathrm{id}}\, & \,\displaystyle\bigoplus_K\, \mathbb{Z}_2\, \ar[r]^-{\mathrm{incl}_\ast} \,& \,H_{m-n+1} (E)
} 
\end{equation}
where \(b_\ast\) maps the \(\mathbb{Z}_2\)-factor labelled by \([\overline{u}]\in K\) identically to the \(\mathbb{Z}_2\)-factor labelled by \(\beta\left([\overline{u}]\right)\) (cf. 4.9). Thus each orbit \(\mathcal{O}\) of \(\beta\) contributes a single \(\mathbb{Z}_2\)-factor to the cokernel of \(b_\ast-\mathrm{id}\) (and hence to the image of \(\mathrm{incl}_\ast\)), and \(\mathrm{incl}_\ast\) just forms the sum of the \(\mathbb{Z}_2\)-entries with labels in \(\mathcal{O}\). In particular, we obtain
\begin{lemma}\label{6.2}
 The image of any element of the form \((0,\ldots,0,1,0,\ldots)\) under \(\mathrm{incl}_\ast\) is nontrivial.
\end{lemma}
Next note the direct sum decomposition
\begin{equation}\label{6.3}
  \mathbb{Z}^n \cong \left(\mathbb{Z}^n\cap \overline{L}\left(\mathbb{R}^m\right)\right) \oplus \mathbb{Z};
\end{equation}
here the \(\mathbb{Z}\)-factor is generated by a vector \(\overline{y}_2\) in \(\mathbb{Z}^n\) which realizes the 
minimal strictly positive distance between elements of \(\mathbb{Z}^n\) and the hyperplane \(\overline{L}\left(\mathbb{R}^m\right)\) in \(\mathbb{R}^n\).

Since \(\overline{L}\circ\overline{A}_M = \overline{A}_N\circ\overline{L}\) (compare proposition 4.2), \(\overline{A}_N\) preserves \(\mathbb{Z}^n\cap \overline{L}\left(\mathbb{R}^m\right)\) and induces the automorphism \(\mathrm{a}\cdot\mathrm{identity}\) on the quotient \(\mathbb{Z}^n \diagup \left( \mathbb{Z}^n \cap \overline{L}\left(\mathbb{R}^m \right)\right) \cong \mathbb{Z}\) where \(a=\pm 1\).
\begin{theorem}\label{6.4}
 Assume \(a=+ 1\). We may choose vectors \(\,\;\overline{w}_1,\ldots,\overline{w}_{n-1}\in\mathbb{Z}^n\) which generate \(\overline{L}\left(\mathbb{Z}^m\right)\). Then\[
\mathrm{MCC}_B \left(f_{L,\overline{v}}, f_0\right)=\mathrm{N}_B \left(f_{L,\overline{v}}, f_0\right) = \left|\det\left(\overline{v},\overline{w}_1,\ldots,\overline{w}_{n-1}\right)\right|.\]
This is also equal to the Reidemeister number \(\,\,\# \mathrm{R}_B\!\left(f_{L,\overline{v}}, f_0\right)\) whenever \(\overline{v}\notin \overline{L}\left(\mathbb{R}^m\right)\); otherwise the Reidemeister set of \(\left(f_{L,\overline{v}}, f_0\right)\) is infinite.
\end{theorem}
\begin{proof}
 Put \(f:= f_{L,\overline{v}}\) for short. If \(\overline{v}\in \overline{L}\left(\mathbb{R}^m\right)\) we can push the zero section \(s_0\) of \(p_N\) slightly away from the image of \(f\) (which is a cooriented hypermanifold in \(N\)). Thus the pair \((f, f_0)\) is loose over \(B\), its minimum and Nielsen numbers vanish, and so does the determinant of the \(n\times n\)-matrix formed by the vectors \(\overline{v},\overline{w}_1,\ldots,\overline{w}_{n-1}\). Furthermore, \(\overline{A}_N\) and \(\beta\) (cf. 4.9) preserve the \(\mathbb{Z}\)-levels in \ref{6.3}; thus there are infinitely many orbits of \(\beta\) or, equivalently, Reidemeister classes (cf. theorem 4.1(i)).

Next consider the case where \(\overline{v}\notin \overline{L}\left(\mathbb{R}^m\right)\). Here the derivatives of the map \(\overline{\ell}: \mathbb{R}^m\times I\, \to \,\mathbb{R}^n\), defined by\[
\overline{\ell}\left(\overline{x},t\right)=\overline{L}(\overline{x})+t\,\overline{v},\]
are everywhere surjective. Therefore the coincidence manifold (which has the form
\begin{equation}\label{6.5}
 \begin{split}
  C(f,f_0)&=\left(q_m\times \mathrm{id}_I\right)\left(\overline{\ell}^{-1}\left(\mathbb{Z}^n\right)\right) \diagup(x,1)\sim\left(A_M(x),0\right)\\
&= \displaystyle\coprod_{\overline{u}\in \mathbb{Z}^n\cap P} \left(q_m\times \mathrm{id}_I\right) \left(\overline{\ell}^{-1}\left\{\overline{u}\right\}\right)\diagup\sim
 \end{split}
\end{equation}
where
\begin{equation}\label{6.6}
 P:=\left\{t_0\overline{v}+t_1\overline{w}_1+\ldots+t_{n-1}\overline{w}_{n-1}\,|\,\text{ all }\, t_i\in[0,1)\right\})
\end{equation}
satisfies the necessary transversality condition and can be used to calculate \(\widetilde{\omega}_B(f,f_0)\) (compare the discussion of 1.4 above or in [GK]). 

Given an integer vector \(\overline{u}\) in the paralleliped \(P\) (cf. \ref{6.6}), the pathcomponent \[
C_{\overline{u}}:= \left(q_m\times \mathrm{id}_I\right) \left(\overline{\ell}^{-1}\left\{\overline{u}\right\}\right)\]
of \(C(f,f_0)\) is an affine \((m-n+1)\)-dimensional subtorus of the fiber \(p_M^{-1}\{\tau\} =T^m\times \{\tau\}\) over \(\tau=\left[\frac{i}{\left|\overline{v}_2\right|}\right]\in S^1=I\diagup\!\sim\,~\)  for some integer \(0\leq i < \left|\overline{v}_2\right|\) where \(\overline{v}_2\) denotes the \(\mathbb{Z}\)-component of \(\overline{v}\) (cf. \ref{6.3}). Consider the restricted map\[
\xymatrix@=1.2cm{
f|\, : T^m\times \{\tau\}\,\, \ar[r] &\,\, T^n\times \{\tau\}}\]
and the resulting coincidence data\[
\xymatrix@=2cm{
C(f|,0)\ar[r]^-{\widetilde{g}|}\ &\ E(f|,0)\ar[r]^-{\mathrm{incl}}\ &\ E\left(f_{L,\overline{v}},f_0\right)
}.\]
It follows from theorem 2.3(ii) that \(\widetilde{g}|\) yields a homotopy equivalence from the pathcomponent \(C_{\overline{u}}\) of \(C(f|,0)\) to some pathcomponent of \(E(f|,0)\). The corresponding statement holds for the inclusion map \(\mathrm{incl}\) since all orbits of \(\beta\) are infinite here (use \ref{4.10}i). Therefore \(\,\widetilde{g}\,:\,C(f,f_0)\,\to\,E(f,f_0)\) (cf. \ref{1.4}) as a whole is a homotopy equivalence. Indeed, since \(a=+ 1\) the Reidemeister set \(\pi_0\left(E\left(f_{L,\overline{v}},f_0\right)\right)\) also corresponds bijectively to the elements \(\overline{u}\) of \(\,\,\mathbb{Z}^n \cap P\). It is not hard to check that this is compatible with the decomposition \ref{6.5}. In particular, all pathcomponents of \(E(f,f_0)\) are essential (as detected by \(\mathrm{mod\,2}\) homology \(H_{m-n+1}\)) and all Nielsen classes are pathconnected. Their number equals \(\,\,\#\left(\mathbb{Z}^n\cap P\right)\,\), i.e. the volume of the paralleliped \(P\) in \(\mathbb{R}^n\) which can be described by the indicated determinant.
\end{proof}

Given a vector \(\overline{v}\in\mathbb{Z}^n\), consider its decomposition
\begin{equation}\label{6.7}
 \overline{v}=\overline{v}_1+\overline{v}_2\in\left(\mathbb{Z}^n\cap \overline{L}\left(\mathbb{R}^m\right)\right) \oplus\mathbb{Z}
\end{equation}
(cf. \ref{6.3}); clearly \(\overline{v}_2\) can be described, up to sign, in terms of the Euclidean distance function by
\begin{equation}\label{6.8}
 \left|\overline{v}_2\right| = \frac{\mathrm{dist}\left(\overline{v},\overline{L}\left(\mathbb{R}^m\right)\right)} {\mathrm{dist}\left(\overline{y}_2,\overline{L}\left(\mathbb{R}^m\right)\right)}
\end{equation}
where \(\,\overline{y}_2\,\) generates the \(\,\mathbb{Z}\)-factor.\\
If \(a=-1\) only the residue class\[
[\overline{v}]\in\mathbb{Z}^n\diagup\left(\overline{A}_N -\mathrm{id}\right)\left(\mathbb{Z}^n\right)\]
and hence only the parity of the integer \(\overline{v}_2\) is determined by the fiberwise homotopy class of \(f\) (cf. prop. \ref{4.6}). In fact, if \(\overline{v}_2\) is even we may find a representive of \([\overline{v}]\) which lies in \(\overline{L}\left(\mathbb{R}^m\right)\) (see also the proof below).
\begin{theorem}\label{6.9}
 Assume \(a=-1\). Then the Reidemeister set \(\pi_0\left(E_B\left(f_{L,\overline{v}},f_0\right)\right)\) is infinite.

If \(\overline{v}_2\) is odd, then the pair \(\left(f_{L,\overline{v}},f_0\right)\) is loose and\[
\mathrm{MCC}_B \left(f_{L,\overline{v}},f_0\right) =\mathrm{N}_B\left(f_{L,\overline{v}},f_0\right) =0\]

If \(\,\overline{v}\in\overline{L}\left(\mathbb{R}^m\right)\), consider the selfmap \(\beta|K'\) defined on the finite set\[
K':=\left(\mathbb{Z}^n\cap\overline{L}\left(\mathbb{R}^m\right)\right)\diagup \overline{L}\left(\mathbb{Z}^m\right)\]
by \(\,\,\beta|K'\left([\overline{u}]\right) = \left[\overline{A}_N\left(\overline{u}-\overline{v}\right)\right]\), \(\,[\overline{u}]\in K'\) (compare \ref{1.17}). Then \[
\mathrm{MCC}_B \left(f_{L,\overline{v}},f_0\right)= \mathrm{N}_B\left(f_{L,\overline{v}},f_0\right)\]
equals the number of \textnormal{odd} order orbits of the \(\mathbb{Z}\)-action on \(K'\) determined by \(\beta|\). (For a more geometric interpretation of this ''odd order condition`` see theorem \ref{4.10}(i).)
\end{theorem}
\begin{proof}
 Again put \(f:=f_{L,\overline{v}}\). The Reidemeister set \(\pi_0\left(E_B\left(f,f_0\right)\right)\) has the same cardinality as the set of \textit{all} orbits of the \textit{full} selfmap \(\beta\) of \(K=\mathbb{Z}^n\diagup\overline{L}\left(\mathbb{Z}^m\right)\cong K'\oplus\mathbb{Z}\) (cf. theorem 4.1(i) and \ref{6.3}).
Since \(a=-1\), \(\beta\) induces an involution on \(\mathbb{Z}\) and hence has infinitely many orbits.

The decomposition
\begin{equation}\label{6.3'}\tag{\ref{6.3}'}
 \mathbb{R}^n=\overline{L}\left(\mathbb{R}^m\right)\times \mathbb{R}
\end{equation}
corresponding to \ref{6.3} is compatible with the gluing ismorphism \(\overline{A}_N\) at least to the extend to imply that the sets \(\overline{L}\left(\mathbb{R}^m\right)\times \left\{\frac{i}{2}\right\}\times I,\) \(i=0,1\), project onto disjoint \(1\)-codimensional submanifolds \(N_i\) of \(N\). First we will show that \(s_{\overline{v}}\) (cf. \ref{4.1}) - and hence \(f_{L,\overline{v}}\) - can be deformed into \(N_i\) where \(i\equiv\overline{v}_2 \,\mathrm{mod}\, 2\).

There exists a homotopy\[
\xymatrix@=1.2cm{
h: I\times I\,\, \ar[r] &\,\, \mathbb{R}^n=\overline{L}\left(\mathbb{R}^m\right)\times \mathbb{R}}\]
such that for all \(t,\tau\in I\) we have 
\begin{enumerate}[(i)]
 \item \(h(t,0)=t\,\overline{v}=t\left(\overline{v}_1,\overline{v}_2\right)\);
\item \(h(0,\tau)=\overline{A}_N\left(0,-\tau\cdot\frac{\overline{v}_2}{2}\right)\)\\
\(h(1,\tau)=\left(\overline{v}_1,(2-\tau)\cdot\frac{\overline{v}_2}{2}\right)\); and
\item \(h(I\times \{1\})\subset\overline{L}\left(\mathbb{R}^m\right)\times \left\{\frac{\overline{v}_2}{2}\right\} \).
\end{enumerate}
The resulting homotopy \(H:I\times I\,\to \,\mathbb{R}^n\times I\), \(\ H(t,\tau):= \left(h(t,\tau),t\right)\), ist compatible (\(\mathrm{mod}\mathbb{Z}^n\)) with the gluing isomorphism \(\overline{A}_N\). When composed with the projection to \(N\) this yields a homotopy which deforms \(s_{\overline{v}}\) into a section of \(p_N\) whose image lies in \(N_i\).

If \(\overline{v}_2\) is odd, this induces (in view of proposition \ref{4.6}) also a homotopy which pushes \(f_{L,\overline{v}}\) into \(N_i=N_1\), i.e., away from the image of \(f_0\) which lies in \(N_0\).

If \(\overline{v}_2\) is even we may (and will) assume that - after a similar deformation - \(\overline{v}\) lies in \(\overline{L}\left(\mathbb{R}^m\right)\) (and hence \(f=f_{L,\overline{v}}\) maps into \(N_0\)). Then, given \(\overline{u}\in\mathbb{Z}^n\cap\overline{L}\left(\mathbb{R}^m\right)\),
\begin{equation}\label{6.10}
 C_{\overline{u}}:=\left(q_m\times\mathrm{id}\right) \left(\left\{\left(\overline{x},t\right) \in \mathbb{R}^m\times I\,|\, \overline{L}\,\overline{x}+t\,\overline{v}=\overline{u}\right\}\right)
\end{equation}
is a nonempty connected, \((m-n+2)\)-dimensional submanifold of \(T^m\times I\) with two boundary components \(\partial_0 C_{\overline{u}}\) and \(\partial_1 C_{\overline{u}}\) (at \(t=0\) and \(t=1\), resp.). Any two such submanifolds \(C_{\overline{u}},C_{\overline{u}'}\) are equal or disjoint according as \(\overline{u}-\overline{u}'\) lies in \(\overline{L}\left(\mathbb{Z}^m\right)\) or not. Hence all of them can be labelled by the elements \([\overline{u}]\) of the quotient set \(K'= \left(\mathbb{Z}^n\cap\overline{L}\left(\mathbb{R}^m\right)\right) \diagup \overline{L}\left(\mathbb{Z}^m\right)\).

The coincidence locus\[
C(f,f_0)\subset M=T^m\times I\diagup\!\sim\]
is obtained from the union of all manifolds \(C_{[\overline{u}]}\) by gluing (via \(\,A_M\), cf. \ref{1.14}) each top end \(\partial_1 C_{[\overline{u}]}\,= \,q_m\left(\overline{L}^{-1}\left(\left\{\overline{u}-\overline{v}\right\}\right)\right)\times \{1\}\) to the corresponding bottom end \(\partial_0 C_{\beta[\overline{u}]} = \,q_m\left(\overline{L}^{-1}\left(\left\{\overline{A}_N\left(\overline{u}-\overline{v}\right)\right\}\right)\right)\times \{0\}\) of \(C_{\overline{A}_N (\overline{u}-\overline{v})}\) (compare \ref{1.16} and \ref{4.10}).

The pair \(\left(f=f_{L,\overline{v}},f_0\right)\) has one big flaw: it does not yet satisfy the transversality condition needed for the construction of \(\omega\)-invariants (cf. [GK], 1.4). To make up for this, let us approximate \(s_{\overline{v}}\) by a smooth section \(s\) of \(p_N\) of the form
\begin{equation}\label{6.11}
 s\left([t]\right)=\left[q_n\left(\overline{s}_1(t),\overline{s}_2(t)\right),t\right],\quad t\in I,
\end{equation}
(compare \ref{4.1} and \ref{6.3'}) where \(\overline{s}_1(0)=0\,\) and (for some small \(\varepsilon>0\)) the values of \(\overline{s}_2(t)\) are \(t, 1-t\) and near \(\varepsilon\), resp., when \(t\) lies in the intervals \([0,\varepsilon],\,[1-\varepsilon,1]\) and \([\varepsilon,1-\varepsilon]\), resp. Then \(f\) is homotopic to the fiber preserving map \(f_s\) defined by
\begin{equation}\label{6.12}
 f_s\left([x,t]\right)=\left[L\,x+s(t),t\right]
\end{equation}
(compare 4.4). Moreover the pair \((f_s,f_0)\) satisfies the desired transversality condition (since \(a=-1\)).

Since \(f_0(M)\) lies in \(N_0=\left(L\left(T^m\right)\times I\right)\diagup\!\sim\ \) but \(s\left([t]\right)\) and \(f_s\left([x,t]\right)\) do so only when \([t]=[0]\in \left(I\diagup\!\sim\right)\ =\ S^1\), the coincidene manifold of \((f_s,f_0)\) is the intersection of \(C(f,f_0)\) with the fiber of \(p_M\) over \([0]=[1]\), i.e. \(C(f_s,f_0)\) is the (image of the) disjoint union of all the affine subtori \(\partial_0 C_{[\overline{u}]},\,[\overline{u}]\in K'\), in the torus \(p_M^{-1}\left\{[0]\right\}\subset M\) (compare the discussion of \ref{6.5}).

If \(\beta\left([\overline{u]}\right)\neq [\overline{u}]\) (compare the definitions in 4.9 and theorem \ref{6.9}), then \(\partial_0 C_{[\overline{u}]}\) and \(\partial_1 C_{[\overline{u}]}\approx\partial_0 C_{\beta[\overline{u}]}\) are disjoint in \(M\) and connected by the bordism \(C_{[\overline{u}]}\) which embeds naturally in \(M\).\\
Deform \(f_s\) (cf. \ref{6.12} and \ref{6.11}) in a small tubular neighbourhood of this bordism until the contribution of \(\overline{s}_2\) is strictly negative along all of \(C_{[\overline{u}]}\) but leave the contributions of \(\overline{s}_1\) and \(L\) unchanged. This eliminates the two coincidence components \(\partial_0 C_{[\overline{u}]}\) and \(\partial_0 C_{\beta[\overline{u}]}\) without creating new ones.

Now recall from theorem 4.1(i) that any given Nielsen class in \(C_B(f_s,f_0)\) consists of all those manifolds \(\partial_0 C_{[u]}\) which are labelled by the elements \([u]\) of the corresponding orbit of \(\beta|K'\). We can reduce it - via pairwise elimination as above - to a single component (or the empty set, resp.) if the cardinality of the orbit is odd (or even, resp.). This geometric fact is closely reflected by the algebra of the upper exact sequence in 5.4ii. In particular, each odd order orbit of \(\beta|K'\) corresponds to a Nielsen class which is both pathconnected and essential; the latter property is already detected by homology with coefficients in \(\mathbb{Z}_2\) (use lemma \ref{6.2}). This establishes the last claim in theorem \ref{6.9}.
\end{proof}
In the spirit of the previous proof we obtain also a geometric essentiality criterion.
\begin{corollary}\label{6.13}
 Assume that \(a=-1\). 

Then a pathcomponent \(Q\) of \(E_B\left(f_{L,\overline{v}},f_0\right)\) is essential if and only if the (restricted) composite projection (cf \ref{1.2} and \ref{1.4})\[
\xymatrix@=1.2cm{
p_M\circ\mathrm{pr}|\,:\, Q\,\,\ar[r] &\,\,B=S^1}\]
induces a nontrivial homomorphism from \(H_1\left(Q;\mathbb{Z}_2\right)\) to \(H_1\left(S^1;\mathbb{Z}_2\right)\).
\end{corollary}
\begin{proof}
According to theorem \ref{4.10}(i) this induced homomorphism is onto precisely if the orbit \(\,\mathcal{O}=\{\ldots,\, [u],\,\beta[u],\ldots\}\,\) which corresponds to \(Q\) has an odd order. This can happen only if \(\,\overline{v}_2\equiv 0(2)\,\) (so that we may assume \(\overline{v}\) to lie in \(\,\overline{L}\left(\mathbb{R}^m\right)\)) and \(\,\mathcal{O}\subset K'\,\) (cf. theorem \ref{6.9}); indeed, when \(\beta\) alternates between different \(\mathbb{Z}\)-levels in \(\,\mathbb{Z}^n\diagup\!\overline{L}\left(\mathbb{Z}^m\right)\,\cong\,K'\oplus\mathbb{Z}\,\), the resulting orbits have an even cardinality. Our claim follows now from the proof fo theorem \ref{6.9}.
\end{proof}
Here is another, more general essentiality criterion in terms of \(\mathrm{mod}\,2\) homology.
\begin{theorem}\label{6.14}
 Whether \(a=+1\) or \(a=-1,\ \) a pathcomponent \(Q\) of \(E_B(f,f_0)\) is essential if and only if\[
\mu_2\left(\widetilde{\omega}_{B,Q}(f,f_0)\right)= \widetilde{g}_\ast\left(\left[C_Q\right]\right)\in H_{m-n+1}\left(Q;\mathbb{Z}_2\right)\]
does not vanish (where \(\mu_2\) is the Hurewicz homomorphism, cf. \ref{1.20}, and \[                                           \widetilde{\omega}_{B,Q}(f,f_0):= \left[C_Q,\,\widetilde{g}|C_Q,\, \overline{g}|\right]\in \Omega_{m-n+1}\left(Q;\widetilde{\varphi}|\right)\]
denotes the contribution of \(Q\) to \(\widetilde{\omega}_B(f,f_0)\), compare definition \ref{1.7}).
\end{theorem}
\begin{proof}
 If \(Q\) is essential, then - in each of the cases discussed in the previous proofs - the corresponding component \(C_Q\) of the (possibly modified) coincidence set is an affine subtorus of the fiber \(p_M^{-1}\left\{[t]\right\}\cong T^m\) for some \([t]\in S^1\). Moreover according to theorem 2.3(ii) \(\widetilde{g}|C_Q\) is a homotopy equivalence to some pathcomponent of \(E|=E_B(f,f_0)|T^m\). Thus \(\mu_2\left(\widetilde{\omega}_{B,Q}(f,f_0)\right)\) has the form \(\mathrm{incl}_\ast(0,\ldots,0,1,0,\ldots)\) (cf. \ref{6.1}) and must be nontrivial by lemma \ref{6.2}. This holds still true in the case \(a=-1\), \(\overline{v}\in\overline{L}\left(\mathbb{R}^m\right)\) where we deformed \((f,f_0)\) into \((f_s,f_0)\). Indeed, such a homotopy induces a fiber homotopy equivalence from \(E_B(f,f_0)\) to \(E_B(f_s,f_0)\) (compare [Ko2]) which is compatible with the arguments in our discussion.
\end{proof}
\begin{theorem}\label{6.15}
For each fiberwise map \(f:M\,\to\, N\)\[
\mathrm{MC}_B(f,f_0)=
\begin{cases}
 \mathrm{N}_B(f,f_0) & \text{if }\, \mathrm{N}_B(f,f_0)=0\, \text{ or }\, m<n;\\
\infty & \text{else.}
\end{cases}
 \]
\end{theorem}
\begin{proof}
 Suppose \((f,f_0)\) is homotopic to a pair \((f',f_0')\) with only a finite set \(C'\) of coincidence points. Then - after a further, small homotopy - we obtain a generic pair \((f_1,f_2)\) with coincidence data \(\left(C,\widetilde{g},\overline{g}\right)\) (cf. \ref{1.3}-\ref{1.5}) such that the coincidence manifold \(C\) (and the paths of \(\widetilde{g}(C)\), resp.) lie in small ball neighbourhoods of the points of \(C'\) in \(M\) (and of \(f'(C')=f_0'(C')\) in \(N\), resp.). Hence \(\widetilde{g}\) is homotopic to a locally constant map.

If \(m-n+1>0\), then \[
\widetilde{g}_\ast\left(\left[C\right]\right)=\mu_2\left(\widetilde{\omega}_B(f_1,f_2)\right) \in H_{m-n+1}\left(E_B(f_1,f_2)\right)\]
vanishes and so does \(\mathrm{N}_B(f_1,f_2)=\mathrm{N}_B(f,f_0) =\mathrm{MCC}_B(f,f_0)=\mathrm{MC}_B(f,f_0)\) in view of theorems \ref{6.14}, \ref{6.4} and \ref{6.9}.

If \(m-n+1=0\), then each essential Nielsen class can be realized by a single point (cf. the proofs of theorems \ref{6.4} and \ref{6.9}); therefore \(\mathrm{MC}_B(f,f_0)=\mathrm{N}_B(f,f_0)\).

These numbers vanish whenever \(m-n+1<0\).
\end{proof}
We have now established theorems \ref{1.13} and \ref{1.18} of the introduction. So let us turn to the proof of theorem \ref{1.22neu}.

First consider the subcase 1\(-\) in theorem \ref{1.18}. If \(\mathrm{quot}([\overline{v}])=0\) (i.e. \(\overline{v}_2\) is even) the Reidemeister invariant \(\,\varrho(f_1,f_2)\,\) (cf. \ref{1.26neu3}) can be represented by a selfmap \(\beta_{\overline{v}}\) of \[
\mathbb{Z}^n\diagup\overline{L}\left(\mathbb{Z}^m\right) \;\; \cong \;\; K' \times \mathbb{Z}\]
where \(\,K'=\left(\mathbb{Z}^n\cap\overline{L}\left(\mathbb{R}^m\right)\right)\diagup \overline{L}\left(\mathbb{Z}^m\right)\,\) (cf. \ref{6.3}, \ref{6.7} and theorem \ref{6.9}) and \(\,\overline{v}\in\overline{L}\left(\mathbb{Z}^m\right)\). Since \(\,\beta_{\overline{v}}\,\) maps \(\,K'\times \{i\}\,\) to \(\,K'\times\{-i\}\,\) for all \(i\in\mathbb{Z}\), the only odd order orbits can lie in \(K'\times \{0\}\) and all the other orbits have an even cardinality. If \(\overline{v}_2\) is odd, then \(\beta_{\overline{v}}\) restricts to a bijection between the two disjoint sets \(\,K'\times\{i\}\,\) and \(\,K'\times \{\overline{v}_2-i\}\,\) and there are no infinite or odd order orbits at all.

Next consider subcase 1+. If \(\overline{v}_2\neq 0\) then all orbits of \(\beta_{\overline{v}}\) are infinite and they can be indexed by the integer vectors in the paralleliped \(P\) (cf. \ref{6.6}). If \(\overline{v}_2=0\) and hence \(\beta_{\overline{v}}\) preserves the \(\mathbb{Z}\)-levels in \(\,K'\times\mathbb{Z}\,\), then every orbit in \(\,K'\times\{i\},\,i\in\mathbb{Z},\) gives rise to infinitely many ``parallel'' orbits in \(\,K'\times\left(i+\left(\#K'\right)\cdot\mathbb{Z}\right)\,\) which have the same (finite!) order; indeed write \(\,\overline{A}_{N\ast}(0,1)\,=\,(\kappa,1)\,\) in \(\,K'\times\mathbb{Z}\,\) and check how\[
\beta_{\overline{v}}(x,i)\;=\;\left(\beta_{\overline{v}}(x)+i\,\kappa\, , i\right),\quad x\in K',\]
depends on \(i\), when taken modulo the order \(\#K'\) of \(K'\).

Table \ref{6.16} sums up these observations concerning the number of orbits of the selfmap \(\beta_{\overline{v}}\) (which represents the Reidemeister invariant \(\,\varrho_B(f_1,f_2)\); cf. \ref{1.20neu}-\ref{1.26neu3}).
\refstepcounter{equation}
\begin{table}[h]\label{6.16}
\begin{tabular}{l||l|l|l||l}
 & \(\nu_{\mathrm{odd}}\) & \(\nu_{\mathrm{even}}\) & \(\nu_\infty\) & \(\nu_B=\nu'_{\mathrm{odd}}+\nu'_{\mathrm{even}}+\nu_\infty\) \\
\hline\hline
Case 0 & finite & finite & 0 & \(\#\{\text{all orbits}\}\quad\quad\;\,\neq 0\) \\
\hline\hline
Subcase 1+, \(\,\overline{v}_2\neq0\) & 0 & 0 & \(\#\left(\mathbb{Z}^n\cap P\right)\) & \(\#\left(\mathbb{Z}^n\cap P\right)\) (cf. \ref{6.6}) \(\neq0\) \\
\hline
Subcase 1+, \(\,\overline{v}_2=0\) & 0 or \(\infty\) & 0 or \(\infty\) & 0 & 0 \\
\hline\hline
Subcase 1-, \(\,\overline{v}_2\,\mathrm{odd}\) & 0 & \(\infty\) & 0 & 0 \\
\hline
Subcase 1-, \(\,\overline{v}_2\,\mathrm{even}\) & finite & \(\infty\) & 0 & \(\nu_{\mathrm{odd}}\)
\end{tabular}
\caption{The number of orbits of \(\varrho_B(f_1,f_2)\) (The pair \((f_1,f_2)\) is loose over \(S^1\) precisely if \(\nu_B=0\).)}
\end{table}

According to theorem \ref{1.18} \(\;\mathrm{MCC}_B(f_1,f_2)=\mathrm{N}_B(f_1,f_2)\;\) agrees with the value of \(\nu_B\left(\varrho_B(f_1,f_2)\right)\) listed in table \ref{6.16}. In view of theorem \ref{1.13} this completes the proof of theorem \ref{1.22neu}.\(\hfill\Box\)\\

A simple way to obtain interesting illustrations for the results of this section is to start from a linear map \(\overline{L}\,:\,\mathbb{Z}^m\,\to\,\mathbb{Z}^n\) of rank \(n\) and compose it with the inclusion map
\begin{equation}
 i_n\;:\;\mathbb{Z}^n\;\subset\;\mathbb{Z}^{n+1}.
\end{equation}
(compare also example \ref{1.20neu2}).
\begin{example}\label{6.18}
 Start from the situation in example \ref{4.12neu} but add an extra Klein bottle \(K\) fiberwise to the target bundle \(N\) and replace \(\overline{L}\) and \(\overline{v}\ \) by \(\,i_n\circ\overline{L}\,:\,\mathbb{Z}^n\,\to\,\mathbb{Z}^{n+1}\) and \(\,i_n(\overline{v})\,\), resp. Then we are in subcase 1-. This transition preserves only \(\nu_{\mathrm{odd}}\); the extra space in \(\,N\,\times_B\,K\,\) makes Nielsen coincidence classes inessential whenever they correspond to an orbit having an even cardinality. \(\hfill\Box\)
\end{example}


\end{document}